\documentclass[12pt,oneside]{article}
\usepackage{amsmath,amssymb,amsfonts,amsthm}
\usepackage{color}

\newcommand{\T}{{\cal T}}

\newcommand{\Real}{\mathbb R}

\newcommand{\set}[1]{\left\{#1\right\}}

\newcommand{\To}{\longrightarrow}

\newcommand{\p}{\pi^{-1}(TM)}
\newcommand {\cp}{\mathfrak{X}(\pi (M))}
\newcommand {\ccp}{\mathfrak{X}^{*}(\pi (M))}
\newcommand {\cpp}{\mathfrak{X}(\T M)}

\def\Section#1{\vspace{30truept}\addtocounter{section}{1}\setcounter{thm}{0}\setcounter{equation}{0}
{\noindent\Large\bf\arabic{section}.~~#1}\par \vspace{12pt}}

\newtheorem{thm}{Theorem}[section]
\newtheorem{cor}[thm]{Corollary}
\newtheorem{lem}[thm]{Lemma}
\newtheorem{prop}[thm]{Proposition}
\newtheorem{defn}[thm]{Definition}

\newtheorem{rem}[thm]{Remark}

\numberwithin{equation}{section}

\begin{document}
\title{{\bf  CONCURRENT $\pi$-VECTOR FIELDS AND ENERGY $\beta$-CHANGE} }
\author{{\bf Nabil L. Youssef$^{\dag}$, S. H. Abed$^{\dag}$ and A. Soleiman$^{\ddag}$}}
\date{}

\maketitle                     
\vspace{-1.15cm}
\begin{center}
{$^{\dag}$Department of Mathematics, Faculty of Science,\\ Cairo
University, Giza, Egypt}
\end{center}
\vspace{-0.8cm}
\begin{center}
nlyoussef2003@yahoo.fr,\ sabed52@yahoo.fr
\end{center}
\vspace{-0.7cm}
\begin{center}
and
\end{center}
\vspace{-0.7cm}
\begin{center}
{$^{\ddag}$Department of Mathematics, Faculty of Science,\\ Benha
University, Benha,
 Egypt}
\end{center}
\vspace{-0.8cm}
\begin{center}
soleiman@mailer.eun.eg
\end{center}
\smallskip

\vspace{1cm} \maketitle
\smallskip

\noindent{\bf Abstract.}  The present paper deals with an
\emph{intrinsic} investigation  of the notion of a concurrent
$\pi$-vector field on the pullback bundle of a Finsler manifold
$(M,L)$. The effect of the existence of a concurrent $\pi$-vector
field on some important special Finsler spaces is studied. An
intrinsic investigation of a particular $\beta$-change, namely the
energy $\beta$-change ($\widetilde{L}^{2}(x,y)=L^{2}(x,y)+
B^{2}(x,y)\, with  \ B:=g(\overline{\zeta},\overline{\eta})$;
\,$\overline{\zeta} $ being a concurrent $\pi$-vector field), is
established. The relation between the two  Barthel connections
$\Gamma$ and $\widetilde{\Gamma}$, corresponding to this change, is
found. This relation, together with the fact that  the Cartan and
the Barthel connections have the same horizontal and vertical
projectors, enable us  to study the energy $\beta$-change of the
fundamental linear connection in Finsler geometry: the Cartan
connection, the Berwald connection, the Chern connection and the
Hashiguchi connection. Moreover, the change
 of their curvature tensors is concluded.
\par
It should be pointed out that the present work is formulated in a
prospective modern coordinate-free form. \footnote{ArXiv Number:
0805.2599}

\bigskip
\medskip\noindent{\bf Keywords:\/}\, Special Finsler space, Pullback bundle, Energy $\beta$-change, Concurrent $\pi$-vector field,
 Canonical spray, Barthel connection,  Cartan
connection, Berwald connection, Chern  connection, Hashiguchi
connection.

\bigskip
\medskip\noindent{\bf 2000 AMS Subject Classification.\/} 53C60, 53B40.
\newpage

\vspace{30truept}\centerline{\Large\bf{Introduction}}\vspace{12pt}
\par
An important aim of Finsler geometry is the construction of a
natural geometric framework of variational calculus and the creation
of geometric models that are appropriate for dealing with different
physical theories, such as general relativity, relativistic optics,
particle physics and others. As opposed to Riemannian geometry, the
extra degrees of freedom offered by Finsler geometry, due to the
dependence of its geometric objects on the directional arguments,
make this geometry potentially more suitable for dealing with such
physical theories at a deeper level.
\par
 Studying Finsler geometry, however, one encounters
substantial difficulties trying to seek analogues of classical
global, or sometimes even local, results of Riemannian geometry.
These difficulties arise mainly from the fact that in Finsler
geometry all geometric objects depend not only on positional
coordinates, as in Riemannian geometry, but also on directional
arguments.
\par
 In Riemannian geometry, there is a canonical linear connection on the
manifold $M$, whereas in Finsler geometry there is a corresponding
canonical linear connection due to E. Cartan. However, this is not a
connection on $M$ but is a connection on $T(\T M)$, the tangent
bundle of $\,\T M$, or on $\pi^{-1}(TM) $, the pullback of the
tangent bundle $TM$ by $\pi: \T M\longrightarrow M$.
\par
The concept of a concurrent vector field in Riemannian geometry had
been introduced and studied by K. Yano  \cite{con.1}. On the other
hand, the notion of a concurrent  vector field in  Finsler geometry
had been studied \emph{locally} by S. Tachibana \cite{con.2}, M.
Matsumoto and K. Eguchi \cite{r90} and others.
\par
In this paper, we study \emph{intrinsically} the notion of a
concurrent $\pi$-vector field on the pullback bundle $\,\pi^{-1}(TM)
$ of a Finsler manifold $(M,L)$. Some properties of concurrent
$\pi$-vector fields are discussed. These properties, in turn, play a
key role in obtaining other interesting results. The effect of the
existence of a concurrent $\pi$-vector field on  some important
special Finsler spaces is investigated.
\par
The infinitesimal transformations (changes) in Finsler geometry are
 important, not only in differential geometry,  but also in
application to other branches of science, especially in the process
of geometrization of physical theories \cite{ny1}. For this reason,
we investigate intrinsically a particular $\beta$-change, which will
be referred to as an energy $\beta$-change:
$$\widetilde{L}^{2}(x,y)=L^{2}(x,y)+ B^{2}(x,y),$$
 where $(M,L)$  is a Finsler manifold admitting  a concurrent
 $\pi$-vector field $\overline{\zeta}$ and $B:=g(\overline{\zeta},\overline{\eta})$; $\overline{\eta}$
 being the fundamental $\pi$-vector field.
Moreover, the relation
 between the two  Barthel
connections $\Gamma$ and $\widetilde{\Gamma}$, corresponding to this
change, is obtained. This relation, together with the fact that  the
Cartan and the Barthel connections have the same horizontal and
vertical projectors, enable us  to study the energy $\beta$-change
of the fundamental linear connections on the pullback bundle of a
Finsler manifold, namely, the Cartan connection, the Berwald
connection, the Chern connection and the Hashiguchi connection.
Moreover, the change
 of their curvature tensors is concluded.
\par
Finally, it should be pointed out that a global formulation of
different aspects of Finsler geometry may give more insight into the
infrastructure of physical theories and helps better understand the
essence of such theories without being trapped into the
complications of indices. This is one of the motivations of the
present work, where all results obtained are formulated in a
prospective modern coordinate-free form.

\Section{Notation and Preliminaries}

In this section, we give a brief account of the basic concepts
 of the pullback approach to intrinsic Finsler geometry necessary for this work. For more
 detail, we refer to \cite{r58},\,\cite{r61} and~\,\cite{r44}.
 We assume, unless otherwise stated, that all geometric objects treated are of class
$C^{\infty}$. The
following notation will be used throughout this paper:\\
 $M$: a real paracompact differentiable manifold of finite dimension $n$ and of
class $C^{\infty}$,\\
 $\mathfrak{F}(M)$: the $\Real$-algebra of differentiable functions
on $M$,\\
 $\mathfrak{X}(M)$: the $\mathfrak{F}(M)$-module of vector fields
on $M$,\\
$\pi_{M}:TM\longrightarrow M$: the tangent bundle of $M$,\\
$\pi^{*}_{M}:T^{*}M\longrightarrow M$: the cotangent bundle of $M$,\\
$\pi: \T M\longrightarrow M$: the subbundle of nonzero vectors
tangent to $M$,\\
$V(TM)$: the vertical subbundle of the bundle $TTM$,\\
 $P:\pi^{-1}(TM)\longrightarrow \T M$ : the pullback of the
tangent bundle $TM$ by $\pi$,\\
$P^{*}:\pi^{-1}(T^{*}M)\longrightarrow \T M$ : the pullback of the
cotangent bundle $T^{*}M$ by $\pi$,\\
 $\mathfrak{X}(\pi (M))$: the $\mathfrak{F}(\T M)$-module of
differentiable sections of  $\pi^{-1}(T M)$,\\
 $\mathfrak{X}^{*}(\pi (M))$: the $\mathfrak{F}(\T M)$-module of
differentiable sections of  $\pi^{-1}(T^{*} M)$,\\
$ i_{X}$ : the interior product with respect to  $X
\in\mathfrak{X}(M)$,\\
$df$ : the exterior derivative  of $f\in \mathfrak{F}(M)$,\\
$ d_{L}:=[i_{L},d]$, $i_{L}$ being the interior product with
respect to a vector form $L$.

\par Elements  of  $\mathfrak{X}(\pi (M))$ will be called
$\pi$-vector fields and will be denoted by barred letters
$\overline{X} $. Tensor fields on $\pi^{-1}(TM)$ will be called
$\pi$-tensor fields. The fundamental $\pi$-vector field is the
$\pi$-vector field $\overline{\eta}$ defined by
$\overline{\eta}(u)=(u,u)$ for all $u\in \T M$.

We have the following short exact sequence of vector bundles,
relating the tangent bundle $T(\T M)$ and the pullback bundle
$\pi^{-1}(TM)$:\vspace{-0.1cm}
$$0\longrightarrow
 \pi^{-1}(TM)\stackrel{\gamma}\longrightarrow T(\T M)\stackrel{\rho}\longrightarrow
\pi^{-1}(TM)\longrightarrow 0 ,\vspace{-0.1cm}$$
 where the bundle morphisms $\rho$ and $\gamma$ are defined respectively by
$\rho := (\pi_{\T M},d\pi)$ and $\gamma (u,v):=j_{u}(v)$, where
$j_{u}$  is the natural isomorphism $j_{u}:T_{\pi_{M}(v)}M
\longrightarrow T_{u}(T_{\pi_{M}(v)}M)$. The vector $1$-form $J$ on
$TM$ defined by $J:=\gamma\circ\rho$ is called the natural almost
tangent structure of $T M$. The vertical vector field $\mathcal{C}$
on $TM$ defined by $\mathcal{C}:=\gamma\circ\overline{\eta} $ is
called the canonical or Liouville vector field.

Let $D$ be  a linear connection (or simply a connection) on the
pullback bundle $\pi^{-1}(TM)$.
 We associate with
$D$ the map \vspace{-0.1cm}
$$K:T \T M\longrightarrow \pi^{-1}(TM):X\longmapsto D_X \overline{\eta}
,\vspace{-0.1cm}$$ called the connection (or the deflection) map of
$D$. A tangent vector $X\in T_u (\T M)$ is said to be horizontal if
$K(X)=0$ . The vector space $H_u (\T M)= \{ X \in T_u (\T M) :
K(X)=0 \}$ of the horizontal vectors
 at $u \in  \T M$ is called the horizontal space to $M$ at $u$  .
   The connection $D$ is said to be regular if
\begin{equation}\label{direct sum}
T_u (\T M)=V_u (\T M)\oplus H_u (\T M) \qquad \forall u\in \T M .
\end{equation}
\par If $M$ is endowed with a regular connection, then the vector bundle
   maps
\begin{eqnarray*}
 \gamma &:& \pi^{-1}(T M)  \To V(\T M), \\
   \rho |_{H(\T M)}&:&H(\T M) \To \pi^{-1}(TM), \\
   K |_{V(\T M)}&:&V(\T M) \To \pi^{-1}(T M)
\end{eqnarray*}
 are vector bundle isomorphisms.
   Let us denote
 $\beta:=(\rho |_{H(\T M)})^{-1}$,
then \vspace{-0.2cm}
   \begin{align}\label{fh1}
    \rho\circ\beta = id_{\pi^{-1} (TM)}, \quad  \quad
       \beta\circ\rho =\left\{
                                \begin{array}{ll}
                                          id_{H(\T M)} & {\,\, on\,\,   H(\T M)} \\
                                         0 & {\,\, on \,\,   V(\T M)}
                                       \end{array}
                                     \right.\vspace{-0.2cm}
\end{align}
The map $\beta$ will be called the horizontal map of the connection
$D$.
\par According to the direct sum decomposition (\ref{direct
sum}), a regular connection $D$ gives rise to a horizontal projector
$h_{D}$ and a vertical projector $v_{D}$, given by
\begin{equation}\label{proj.}
h_{D}=\beta\circ\rho ,  \ \ \ \ \ \ \ \ \ \ \
v_{D}=I-\beta\circ\rho,
\end{equation}
where $I$ is the identity endomorphism on $T(TM)$: $I=id_{T(TM)}$.
\par
 The (classical)  torsion tensor $\textbf{T}$  of the connection
$D$ is defined by
$$\textbf{T}(X,Y)=D_X \rho Y-D_Y\rho X -\rho [X,Y] \quad
\forall\,X,Y\in \mathfrak{X} (\T M).$$ The horizontal ((h)h-) and
mixed ((h)hv-) torsion tensors, denoted by $Q $ and $ T $
respectively, are defined by \vspace{-0.2cm}
$$Q (\overline{X},\overline{Y})=\textbf{T}(\beta \overline{X}\beta \overline{Y}),
\, \,\, T(\overline{X},\overline{Y})=\textbf{T}(\gamma
\overline{X},\beta \overline{Y}) \quad \forall \,
\overline{X},\overline{Y}\in\mathfrak{X} (\pi (M)).\vspace{-0.2cm}$$
\par
The (classical) curvature tensor  $\textbf{K}$ of the connection $D$
is defined by
 $$ \textbf{K}(X,Y)\rho Z=-D_X D_Y \rho Z+D_Y D_X \rho Z+D_{[X,Y]}\rho Z
  \quad \forall\, X,Y, Z \in \mathfrak{X} (\T M).$$
The horizontal (h-), mixed (hv-) and vertical (v-) curvature
tensors, denoted by $R$, $P$ and $S$ respectively, are defined by
$$R(\overline{X},\overline{Y})\overline{Z}=\textbf{K}(\beta
\overline{X}\beta \overline{Y})\overline{Z},\quad
P(\overline{X},\overline{Y})\overline{Z}=\textbf{K}(\beta
\overline{X},\gamma \overline{Y})\overline{Z},\quad
S(\overline{X},\overline{Y})\overline{Z}=\textbf{K}(\gamma
\overline{X},\gamma \overline{Y})\overline{Z}.$$ The contracted
curvature tensors, denoted by $\widehat{R}$, $\widehat{P}$ and
$\widehat{S}$ respectively, are also known as the
 (v)h-, (v)hv- and (v)v-torsion tensors and are defined by
$$\widehat{R}(\overline{X},\overline{Y})={R}(\overline{X},\overline{Y})\overline{\eta},\quad
\widehat{P}(\overline{X},\overline{Y})={P}(\overline{X},\overline{Y})\overline{\eta},\quad
\widehat{S}(\overline{X},\overline{Y})={S}(\overline{X},\overline{Y})\overline{\eta}.$$
If $M$ is endowed with a metric $g$ on $\p$, we write
\begin{equation}\label{cur.g}
    R(\overline{X},\overline{Y},\overline{Z}, \overline{W}):
=g(R(\overline{X},\overline{Y})\overline{Z}, \overline{W}),\,
\cdots, \, S(\overline{X},\overline{Y},\overline{Z}, \overline{W}):
=g(S(\overline{X},\overline{Y})\overline{Z}, \overline{W}).
\end{equation}

\begin{lem}\label{bracket}{\em\cite{r92}} Let $D$ be a regular connection on $\p$
whose (h)hv-torsion tensor $T$ has the property that
$\,T(\overline{X}, \overline{\eta})=0$. Then, we
have{\em:}\vspace{-0.2cm}
   \begin{description}
 \item[(a)] $[\beta \overline{X},\beta \overline{Y}]=
     \gamma\widehat{R}(\overline{X},\overline{Y})
     + \beta(D_{\beta \overline{X}}\overline{Y}-
     D_{\beta \overline{Y}}\overline{X}-Q(\overline{X},\overline{Y})),$

    \item[(b)] $[\gamma \overline{X},\beta \overline{Y}]=-
     \gamma(\widehat{P}(\overline{Y},\overline{X})+D_
     {\beta \overline{Y}}\overline{X})
     +\beta( D_{\gamma \overline{X}}\overline{Y}-T(\overline{X},\overline{Y})),$

   \item[(c)] $[\gamma \overline{X},\gamma \overline{Y}]=
     \gamma(D_{\gamma \overline{X}}\overline{Y}-
     D_{\gamma \overline{Y}}\overline{X}+\widehat{S}(\overline{X},\overline{Y}))$.
     \end{description}
\end{lem}

The following theorem guarantees the existence  and uniqueness of
the Cartan connection on the pullback bundle.\vspace{-0.2cm}
\begin{thm} {\em\cite{r92}} \label{th.1} Let $(M,L)$ be a Finsler
manifold and  $g$ the Finsler metric defined by $L$. There exists a
unique regular connection $\nabla$ on $\pi^{-1}(TM)$ such
that\vspace{-0.2cm}
\begin{description}
  \item[(a)]  $\nabla$ is  metric\,{\em:} $\nabla g=0$,

  \item[(b)] The (h)h-torsion of $\nabla$ vanishes\,{\em:} $Q=0
  $,
  \item[(c)] The (h)hv-torsion $T$ of $\nabla$\, satisfies\,\emph{:}
   $g(T(\overline{X},\overline{Y}), \overline{Z})=g(T(\overline{X},\overline{Z}),\overline{Y})$.
\end{description}
\end{thm}
\vspace{5pt}
 One can show that the (h)hv-torsion  of the Cartan connection is symmetric and  has the
property that  $T(\overline{X},\overline{\eta})=0$ for all
$\overline{X} \in \mathfrak{X} (\pi (M))$.

\begin{defn}  Let $(M,L)$ be a Finsler
manifold and  $g$ the Finsler metric defined by $L$. We
define\,\emph{:}\vspace{-0.2cm}
\begin{eqnarray*}
\ell(\overline{X})&:=&L^{-1}g(\overline{X},\overline{\eta}),\\
\hbar&:=& g-\ell \otimes \ell: \text{the  angular metric tensor},\\
T(\overline{X},\overline{Y},\overline{Z})&:
=&g(T(\overline{X},\overline{Y}),\overline{Z}): \text{the  Cartan
 tensor},\\
    C(\overline{X})&:=& Tr\{\overline{Y} \longmapsto
T(\overline{X},\overline{Y})\}: \text{the  contracted torsion},\\
 g(\overline{C}, \overline{X})&:=&C(\overline{X}): \text{$\overline{C}$ is the $\pi$-vector field associated with the $\pi$-form $C$},\\
 Ric^v(\overline{X},\overline{Y})&:=& Tr\{
\overline{Z} \longmapsto S(\overline{X},\overline{Z})\overline{Y}\}:
\text{the  vertical Ricci tensor},\\
g(Ric_0^v(\overline{X}),\overline{Y})&:=&Ric^v(\overline{X},\overline{Y}):
\text{the vertical Ricci  map} \  Ric_0^v,\\
 Sc^v&:=&\text {Tr}\{ \overline{X} \longmapsto Ric_0^v(\overline{X})\}:
 \text{the vertical scalar curvature}.
\end{eqnarray*}
\end{defn}

Deicke theorem \cite{nr2} can be formulated globally as
follows:\vspace{-0.2cm}
\begin{lem}\label{.le.2} Let  $(M,L)$ be a Finsler manifold. The following
assertions are  equivalent\,{\em:}
\begin{description}
    \item[(a)] $(M,L)$ is Riemannian,
    \item[(b)] The $(h)hv$-torsion tensor $T$ vanishes,
    \item[(c)] The $\pi$-form $C$ vanishes.
\end{description}
\end{lem}

 Concerning  the Berwald connection on the pullback
bundle, we have\vspace{-0.2cm}
\begin{thm}{\em\cite{r92}} \label{th.1a} Let $(M,L)$ be a Finsler manifold. There exists a
unique regular connection ${{D}}^{\circ}$ on $\pi^{-1}(TM)$ such
that
\begin{description}
 \item[(a)] $D^{\circ}_{h^{\circ}X}L=0$,
  \item[(b)]   ${{D}}^{\circ}$ is torsion-free\,{\em:} ${\textbf{T}}^{\circ}=0 $,
  \item[(c)]The (v)hv-torsion tensor $\widehat{P^{\circ}}$ of ${D}^{\circ}$ vanishes\,\emph{:}
   $\widehat{P^{\circ}}(\overline{X},\overline{Y})= 0$.
  \end{description}
  \par Such a connection is called the Berwald
  connection associated with the Finsler manifold $(M,L)$.
\end{thm}

\begin{thm}\label{2.th.1}{\em{\cite{r92}}} Let $(M,L)$ be a Finsler manifold.
The  Berwald connection $D^{\circ}$   is expressed in terms of the
Cartan  connection $\nabla $ as\vspace{-0.2cm}
  $$D^{\circ}_{X}\overline{Y} = \nabla _{X}\overline{Y}
+\widehat{P}(\rho X, \overline{Y})- T(KX,\overline{Y}), \quad
\forall\: X \in \mathfrak{X} (T M), \, \overline{Y} \in
\mathfrak{X}(\pi(M)) . \vspace{-0.2cm}$$ In particular, we have:
\begin{description}
  \item[(a)] $ D^{\circ}_{\gamma \overline{X}}\overline{Y}=\nabla _{\gamma
  \overline{X}}\overline{Y}-T(\overline{X},\overline{Y})$,

 \item[(b)] $ D^{\circ}_{\beta \overline{X}}\overline{Y}=\nabla _{\beta
  \overline{X}}\overline{Y}+\widehat{P}(\overline{X}, \overline{Y})$.
\end{description}
\end{thm}

\par We terminate this section
by some concepts and results concerning the Klein-Grifone approach
to intrinsic Finsler geometry. For more details, we refer to
\cite{r21}, \cite{r22} and \cite{r27}.

 A semispray  is a vector field $X$ on $TM$,
 $C^{\infty}$ on $\T M$, $C^{1}$ on $TM$, such that
$\rho\circ X = \overline{\eta}$. A semispray $X$ which is
homogeneous of degree $2$ in the directional argument
($[\mathcal{C},X]= X $) is called a spray.

\begin{prop}{\em{\cite{r27}}}\label{pp} Let $(M,L)$ be a Finsler manifold. The vector field
$G$ on $TM$ defined by $i_{G}\,\Omega =-dE$ is a spray, where
 $E:=\frac{1}{2}L^{2}$ is the energy function and $\Omega:=dd_{J}E$.
 Such a spray is called the canonical spray.
 \end{prop}

A nonlinear connection on $M$ is a vector $1$-form $\Gamma$ on $TM$,
$C^{\infty}$ on $\T M$, $C^{0}$ on $TM$, such that
$$J \Gamma=J, \quad\quad \Gamma J=-J .$$
The horizontal and vertical projectors $h_{\Gamma}$\,  and
$v_{\Gamma}$ associated with $\Gamma$ are defined by
   $h_{\Gamma}:=\frac{1}{2} (I+\Gamma)$ and $v_{\Gamma}:=\frac{1}{2}
 (I-\Gamma)$. To each nonlinear connection $\Gamma$ there is associated a
 semispray $S$ defined by $S=h_{\Gamma}S'$, where $S'$ is an
 arbitrary semispray.
 A nonlinear
connection $\Gamma$ is homogeneous if $[\mathcal{C},\Gamma]=0$. The
torsion of a nonlinear connection $\Gamma$ is the vector $2$-form
$t$ on $TM$ defined by $t:=\frac{1}{2} [J,\Gamma]$. The curvature of
$\Gamma$ is the vector $2$-form $\mathfrak{R}$ on $TM$ defined by
$\mathfrak{R}:=-\frac{1}{2}[h_{\Gamma},h_{\Gamma}]$. A nonlinear
connection $\Gamma$ is said to be conservative if
$d_{h_{\Gamma}}\,E=0$.

\begin{thm} \label{th.9a} {\em{\cite{r22}}} On a Finsler manifold $(M,L)$, there exists a unique
conservative homogenous nonlinear  connection  with zero torsion. It
is given by\,{\em:} \vspace{-0.3cm} $$\Gamma =
[J,G],\vspace{-0.3cm}$$ where $G$ is the canonical spray.\\
 Such a nonlinear connection is called the canonical connection, the Barthel connection or the Cartan nonlinear connection
 associated with $(M,L)$.
\end{thm}
\par
It should be noted that the semispray associated with the Barthel
connection is a spray, which is the canonical spray.

\begin{prop}\label{pp.1}{\em{\cite{r62}}} Under a change $L \To \widetilde{L}$
of Finsler structures on $M$,
 the corresponding Barthel connections $\Gamma$ and $\widetilde{\Gamma}
  $  are related by\vspace{-0.2cm}
  \begin{equation}
    \widetilde{\Gamma} = \Gamma -2 {\L},\, \text{with}\, {\L}:= \gamma o N o \rho.\vspace{-0.2cm}
\end{equation}
  Moreover, we have $ \widetilde{h} =h - {\L} ,\:  \widetilde{v} =v +
  {\L}$.
{\em(}the definition of $N$ is found in {\em{\cite{r62}}}{\em )}.
\end{prop}


\Section{Finsler spaces admitting concurrent $\pi$-vector fields}

The notion of a concurrent vector field has been introduced and
investigated in Riemannian geometry by K. Yano {\cite{con.1}}.
Concurrent vector fields have been studied in Finsler geometry by
Matsumoto and Eguchi {\cite{r90}}, Tachibana {\cite{con.2}} and
others. These studies were accomplished by the use of local
coordinates. In this section, we introduce and investigate
\textbf{intrinsically} the notion of a concurrent $\pi$-vector field
in Finsler geometry. The properties of concurrent $\pi$-vector
fields are obtained.
\par
 In what follows $\nabla$ will
denote the Cartan connection associated with a Finsler manifold
$(M,L)$ and $S$, $P$ and $R$ will denote the three crvature tensors
of $\nabla$.

\begin{defn}\label{2.def.1}Let $(M,L)$ be a Finsler manifold.
A $\pi$-vector field
 $\overline{\zeta} \in \cp$ is
called a concurrent $\pi$-vector field if it satisfies the following
conditions\vspace{-0.2cm}
\begin{equation}\label{12.eq.1}
     \nabla_{\beta \overline{X}}\,\overline{\zeta}=- \overline{X} , \qquad \nabla_{\gamma \overline{X}}\,\overline{\zeta}=0.\vspace{-0.2cm}
    \end{equation}
In other words, $\overline{\zeta}$ is a concurrent $\pi$-vector
field if $\nabla_{X}\,\overline{\zeta}=- \rho X $ for all $X\in
\cpp$, or briefly, $\nabla \overline{\zeta}=- \rho $.
\end{defn}

The following two Lemmas are useful for subsequence
use.\vspace{-0.2cm}
\begin{lem}\label{2.le.1} Let $(M,L)$ be a Finsler manifold. If $\overline{\zeta}\in \cp$
is a concurrent $\pi$-vector field and $\alpha \in \ccp$ is the
$\pi$-form  associated with $\overline{\zeta}$ under the duality
defined by the metric $g${\em\,:} $\alpha=i_{\overline{\zeta}}\,g$,
then the $\pi$-form $\alpha$ has the properties\vspace{-0.2cm}
\begin{description}
    \item[(a)] $(\nabla_{\beta \overline{X}}\alpha)(\overline{Y})=-g(\overline{X},\overline{Y})$,

    \item[(b)] $(\nabla_{\gamma\overline{X}}\alpha)(\overline{Y})=0$.
\end{description}
Equivalently, $\nabla_{ X} \alpha=- i_{\rho X}\,g$.
\end{lem}

\begin{proof} As $\nabla_{X}g=0$, we get\\
$
  (\nabla_{ X}\alpha)(\overline{Y}) =
  \nabla_{X}g(\overline{\zeta}, \overline{Y})-g(\overline{\zeta}, \nabla_{X}\overline{Y}) =
  (\nabla_{X}g)(\overline{\zeta}, \overline{Y})+g(\nabla_{X}\overline{\zeta}, \overline{Y}) =
   -g(\rho X, \overline{Y}).
$
\end{proof}

\begin{lem}\label{2.le.2} For every $ X, Y \in \cpp$
 and $ \overline{Z}, \overline{W}\in \cp $, we have\vspace{-0.2cm}
$$g(\textbf{K}( X,  Y)\overline{Z}, \overline{W})=-g(\textbf{K}( X,  Y)\overline{W}, \overline{Z}),$$
where $\textbf{K}$ is the (classical) curvature of the Cartan
connection.
\end{lem}
\begin{proof}
Follows from Lemma 2.4 of \cite {r96} since $\nabla g=0$.
\end{proof}

 Now, we have the following \vspace{-0.2cm}
\begin{prop}\label{2.pp.1}Let $\overline{\zeta} \in \cp$ be
a concurrent $\pi$-vector field on $(M,L)$.\\
For the v-curvature tensor $S$, the following relations
hold\,\emph{:}
\begin{description}
    \item[(a)]$S(\overline{X},\overline{Y})\,\overline{\zeta}=0$, \quad $S(\overline{X}, \overline{Y}, \overline{Z}, \overline{\zeta})=0$.
    \item[(b)]$(\nabla_{\gamma \overline{Z}}S)(\overline{X},\overline{Y}, \overline{\zeta})=0$, \quad
    $(\nabla_{\beta\overline{Z}}S)(\overline{X},\overline{Y}, \overline{\zeta})=
    S(\overline{X},\overline{Y})\overline{Z}$.
     \item[(c)]  $(\nabla_{\beta \overline{\zeta}}S)(\overline{X},\overline{Y}, \overline{\zeta})=0$.
\end{description}
For the hv-curvature tensor $P$, the following relations
hold\,\emph{:}
\begin{description}
    \item[(d)]$P(\overline{X},\overline{Y})\,\overline{\zeta}=-T(\overline{Y},\overline{X})$, \quad
      $P(\overline{X}, \overline{Y}, \overline{Z}, \overline{\zeta})=T(\overline{X}, \overline{Y},\overline{Z})$.
    \item[(e)] $(\nabla_{\gamma \overline{Z}}P)(\overline{X},\overline{Y}, \overline{\zeta})
    =-(\nabla_{\gamma \overline{Z}}T)(\overline{Y}, \overline{X})$,\\    ${\!\!\!\!}(\nabla_{\beta \overline{Z}}P)(\overline{X},\overline{Y}, \overline{\zeta})=-(\nabla_{\beta \overline{Z}}T)(\overline{Y},\overline{X})+
  P(\overline{X},\overline{Y})\overline{Z}$.
   \item[(f)] $(\nabla_{\beta \overline{\zeta}}P)(\overline{X},\overline{Y},
    \overline{\zeta})=-(\nabla_{\beta \overline{\zeta}}T)(\overline{Y},\overline{X})
    -T(\overline{Y},\overline{X})$.
\end{description}
For the h-curvature tensor $R$, the following relations
hold\,\emph{:}
\begin{description}
    \item[(g)]$R(\overline{X},\overline{Y})\,\overline{\zeta}=0$, \quad
     $R(\overline{X}, \overline{Y}, \overline{Z}, \overline{\zeta})=0$.
     \item[(h)]$(\nabla_{\gamma \overline{Z}}R)(\overline{X},\overline{Y}, \overline{\zeta})=0$, \quad
     $(\nabla_{\beta \overline{Z}}R)(\overline{X},\overline{Y}, \overline{\zeta})=
     R(\overline{X},\overline{Y})\overline{Z}$.
 \item[(i)] $(\nabla_{\beta\overline{\zeta}}R)(\overline{X},\overline{Y}, \overline{\zeta})=0$.
\end{description}

\end{prop}

\begin{proof}
The proof follows from the properties of the curvature tensors $S$,
$P$ and $R$ investigated in \cite{r96} together with Definition
\ref{2.def.1} and Lemma \ref{2.le.2}, taking into account the fact
that the (h)h-torsion of the Cartan connection vanishes.
\end{proof}

\begin{cor}\label{2.pp.2}Let $\overline{\zeta} \in \cp$ be
a concurrent $\pi$-vector field. For every $ \overline{X},
\overline{Y} \in \cp$, we have\vspace{-0.2cm}
\begin{description}
   \item[(a)]$T(\overline{X},\overline{\zeta})=T(\overline{\zeta}, \overline{X})=0$,
   \item[(b)] $\widehat{P}(\overline{X},\overline{\zeta})
   =\widehat{P}(\overline{\zeta}, \overline{X})=0$,
 \item[(c)] $P(\overline{X},\overline{\zeta})\overline{Y}=P(\overline{\zeta},\overline{X})\overline{Y}=0$.
\end{description}
\end{cor}

\begin{proof}
~\par \vspace{5pt} \noindent\textbf{(a)} The proof follows from
Proposition \ref{2.pp.1}(d) by setting
$\overline{Z}=\overline{\zeta}$, taking into account the fact that
$g(T(\overline{X}, \overline{Y}),\overline{Z})=g(T(\overline{X},
\overline{Z}),\overline{Y})$ and
$g(P(\overline{X},\overline{Y})\overline{Z},\overline{Z})=0$ (lemma
\ref{2.le.2}) together with the symmetry of $T$.

\vspace{5pt}
 \noindent\textbf{(b)}  Follows from the identity
$\widehat{P}(\overline{X},\overline{Y})=(\nabla_{\beta
\overline{\eta}}T)(\overline{X},\overline{Y})$ \cite{r96}, making
use of (a) and the fact that $T(\overline{X},\overline{\eta})=0.$

\vspace{5pt} \noindent\textbf{(c)}  We have \cite{r96}
\begin{equation}\label{01}
\left.
    \begin{array}{rcl}
   P(\overline{X},\overline{Y},\overline{Z},\overline{W})&=&
   g((\nabla_{\beta\overline{Z}}T)(\overline{Y},\overline{X}),
\overline{W})
   -g((\nabla_{\beta
\overline{W}}T)(\overline{Y},\overline{X}), \overline{Z})
   \\
    &&-g(T(\overline{X},\overline{W}),\widehat{P}(\overline{Z},\overline{Y}))
   +g(T(\overline{X},\overline{Z}),\widehat{P}(\overline{W},\overline{Y})).
\end{array}
  \right.
\end{equation}
From which, by setting $\overline{Y}=\overline{\zeta}$ (resp.
$\overline{X}=\overline{\zeta}$) and using (a) and (b) above, the
result follows.
\end{proof}

\begin{lem}\label{def.ind} Let $(M,L)$ be a Finsler manifold and $ D^{\circ} $
  the Berwald connection on $\p$. Then, we have\vspace{-0.2cm}
 \begin{description}
   \item[(a)] A $\pi$-vector field  $\overline{Y} \in \cp$ is independent of the directional argument
 $y$   if,  and only if,  $D^{\circ}_{\gamma \overline{X}}\overline{Y}=0
 $ for all $\overline{X} \in \cp$,
   \item[(b)] A scalar \emph{(}vector\emph{)} $\pi$-form  $ \omega $ is independent of the directional argument
 $y$ if,  and only if, $D^{\circ}_{\gamma \overline{X}}\, \omega=0$ for all $\overline{X} \in \cp$.
 \end{description}
 \end{lem}

\begin{thm}\label{th.ind}  A concurrent $\pi$-vector field $\overline{\zeta}$ and its associated $\pi$-form
$\alpha$ are independent of the directional argument $y$.
\end{thm}

\begin{proof}
 By Theorem \ref{2.th.1}(a), we have\vspace{-0.2cm}
$$D^{\circ}_{\gamma \overline{X}}\overline{Y}=\nabla_{\gamma \overline{X}}\overline{Y}-T(\overline{X}, \overline{Y})\vspace{-0.2cm}.$$
 From which,  by setting $\overline{Y}=\overline{\zeta}$,
taking into account (\ref{12.eq.1}), Corollary \ref{2.pp.2}(a)  and
Lemma \ref{def.ind}, we conclude that  $\overline{\zeta}$ is
independent of the directional argument.
\par
On the other hand, we have from the above relation\vspace{-0.2cm}
  $$(D^{\circ}_{\gamma \overline{X}}\alpha)(\overline{Y})=(\nabla_{\gamma \overline{X}}\alpha)(\overline{Y})+g(T(\overline{X}, \overline{Y}),
\overline{\zeta}).\vspace{-0.2cm}$$ This, together with Lemma
\ref{2.le.1}(b) and Corollary \ref{2.pp.2}(a), imply that $\alpha$
is also independent of the directional argument.
\end{proof}


\Section{Special Finsler spaces  admitting  concurrent
$\pi$-vector\vspace{7pt} fields}

In this section, we investigate the effect of the existence of a
concurrent $\pi$-vector field on  some important special Finsler
spaces. The intrinsic definitions  of the special Finsler spaces
treated here are quoted from \cite{r86}.

\vspace{7pt}
\par
 For later use, we need the
following  lemma.\vspace{-0.2cm}
\begin{lem}\label{.le.1}Let $(M,L)$ be a Finsler manifold  which admits a concurrent
$\pi$-vector $\overline{\zeta}$. Then, we
have\,\emph{:}\vspace{-0.2cm}
\begin{description}
\item[(a)] The concurrent $\pi$-vector field $\overline{\zeta}$ is everywhere non-zero.

    \item[(b)] The scalar function  $B:=g(\overline{\zeta}, \overline{\eta})$ is everywhere non-zero.
    \item[(c)]The $\pi$-vector field
    $\overline{m}:=\overline{\zeta}-\frac{B}{L^2}\,\overline{\eta}$ is everywhere non-zero and is orthogonal to $\overline{\eta}$.
    \item[(d)]The $\pi$-vector fields $\overline{m}$ and $\overline{\zeta}$ satisfy
    $g(\overline{m},\overline{\zeta})=g(\overline{m},\overline{m})\neq0$.
    \item[(e)] The angular metric tensor $\hbar$ satisfies
 $\hbar(\overline{\zeta},\overline{X})\neq0$ for all $\overline{X}\neq\overline{\eta}$.
\end{description}
\end{lem}

\begin{proof} Property \textbf{(a)} is clear.

\vspace{5pt}
 \noindent\textbf{(b)} If $B:=g(\overline{\zeta},\overline{\eta})=0$,
then
$$
  0= (\nabla_{\gamma \overline{X}}g)(\overline{\zeta},\overline{\eta}) =
  \nabla_{\gamma \overline{X}}g(\overline{\zeta},\overline{\eta})-g(\overline{\zeta}, \overline{X})=
  -g(\overline{\zeta}, \overline{X}),\ \ \forall \  \overline{X}\in \cp,
  $$
 which contradicts (a).

\vspace{5pt}
 \noindent\textbf{(c)} If  $\overline{m}=0$, then $L^{2}\overline{\zeta}-B\overline{\eta}=0$. Differentiating covariantly with respect
to $\gamma \overline{X}$, we get
\begin{equation}\label{.1.eq.1}
2g(\overline{X}, \overline{\eta})\overline{\zeta}-B
\overline{X}-g(\overline{X},\overline{\zeta})\overline{\eta}=0.
\vspace{-0.2cm}
\end{equation}
From which,
\begin{equation}\label{eq..}
    g(\overline{X},\overline{\zeta})=\displaystyle{\frac{B}{L^2}}g(\overline{X},
\overline{\eta}).
\end{equation}
 By  (\ref{.1.eq.1}), using (\ref{eq..}), we obtain
\begin{eqnarray*}
  0&=&2g(\overline{X}, \overline{\eta})g(\overline{Y},\overline{\zeta})-B g(\overline{X},\overline{Y})-g(\overline{X},\overline{\zeta})g(\overline{Y},\overline{\eta}) \\
   &=& 2\frac{B}{L^2} g(\overline{Y}, \overline{\eta})g(\overline{X},\overline{\eta})- B g(\overline{X},\overline{Y})-\frac{B}{L^2} g(\overline{X}, \overline{\eta})g(\overline{Y},\overline{\eta})\\
    &=&- B\{ g(\overline{X},\overline{Y})-\frac{1}{L^2}g(\overline{Y}, \overline{\eta})g(\overline{X},\overline{\eta})\}=- B\hbar(\overline{X},\overline{Y}).
\end{eqnarray*}
From which, since $B\neq0$, we are led to a contradiction:
$\hbar=0$.
\par
On the other hand, the orthogonality of the two
$\pi$-vector fields $\overline{m}$ and $\overline{\eta}$ follows
from the identities $g(\overline{\eta},\overline{\eta})=L^2$ and
$g(\overline{\eta},\overline{\zeta})=B$.

\vspace{5pt}
 \noindent\textbf{(d)} Follows from (c).

 \vspace{5pt}
\noindent\textbf{(e)} Suppose that
$\hbar(\overline{X},\overline{\zeta})=0$ for all
$\overline{X}\neq\overline{\eta}\in\cpp$, then , we have
$$
  0 =(\nabla_{\beta \overline{X}}\hbar)(\overline{Y},\overline{\zeta})=
  \nabla_{\beta \overline{X}}\hbar(\overline{Y},\overline{\zeta})-\hbar
  (\nabla_{\beta \overline{X}}\overline{Y},\overline{\zeta})+\hbar(\overline{X},\overline{Y} )= \hbar(\overline{X},\overline{Y}),
$$
which contradicts the fact that  $\hbar\neq0$.
\end{proof}

\begin{defn}\label{def.1a} A Finsler manifold $(M,L)$ is\,{\em:}
\begin{description}
  \item[(a)]  Riemannian  if the metric tensor $g(x,y)$ is independent
  of $y$ or, equivalently, if
  $$T(\overline{X},\overline{Y})=0,\,\,\,
    \text{for all}\,\,\,\overline{X}, \overline{Y}\in \mathfrak{X}(\pi(M)).$$

  \item[(b)] locally Minkowskian  if the metric tensor $g(x,y)$ is independent
  of $x$ or,  equivalently, if $$\nabla_{\beta \overline{X}}\,T
  =0\,\,\, \text{and}\,\,\,  R=0.$$
      \end{description}
\end{defn}

\begin{defn}\label{7.def.2} A Finsler manifold $(M,L)$ is\,{\em:}
\begin{description}
  \item[(a)] a Berwald manifold if the torsion tensor $T$ is horizontally
  parallel\,{\em:}
  $\nabla_{\beta \overline{X}}\,T =0.$

\item[(b)] a Landsberg manifold  if $\widehat{P}(\overline{X},\overline{Y})=0$, or
 equivalently, if \  $\nabla_{\beta \overline{\eta}}T =0$.

 \item[(c)] a general Landsberg manifold if the trace of the linear map
  $\overline{Y} \longmapsto \widehat{P}(\overline{X},\overline{Y})$ is
  identically zero for all $\overline{X} \in \mathfrak{X} (\pi (M))$,
  or equivalently, if  $\nabla _{\beta \overline{\eta}}\, C =0 $.
\end{description}
\end{defn}

Now, we have\vspace{-0.2cm}
\begin{thm}\label{.thm.1}Let $(M,L)$ be a Finsler manifold  which admits a concurrent $\pi$-vector field $\overline{\zeta} $.
Then, the following assertions are equivalent\,:
\vspace{-0.2cm}
\begin{description}
\item[(a)] $(M,L)$ is a Berwald manifold,

\item[(b)]$(M,L)$ is a Landsberg manifold,

\item[(c)]$(M,L)$ is a Riemannian manifold.
\end{description}
\end{thm}

\begin{proof} The implications $\textbf{(a)}\Longrightarrow \textbf{(b)}$
and $\textbf{(c)} \Longrightarrow\textbf{(a)}$ are trivial. Now, we
prove the implication $\textbf{(b)} \Longrightarrow \textbf{(c)}$.
As $(M,L)$ is a Landsberg manifold, $\widehat{P}=0$. Consequently,
the hv-curvature $P$ vanishes \cite{r96}. Hence,
$0=P(\overline{X},\overline{Y},\overline{Z}, \overline{\zeta})
=T(\overline{X},\overline{Y},\overline{Z})$ by Proposition
\ref{2.pp.1}(d). The result follows then from Deicke theorem (Lemma
\ref{.le.2}).
\end{proof}

\begin{defn}\label{.def.1}A Finsler manifold $(M,L)$ is said to be{\em{\,\!:}}
\begin{description}
    \item[(a)]  $C^h$-recurrent  if the (h)hv-torsion tensor
 $T$  satisfies the condition\\
  $\nabla_{\beta \overline{X}}\,T
 =\lambda_{o} (\overline{X})\,T,$
 where $\lambda_{o}$ is a $\pi$-form  of order one.

    \item[(b)]$C^v$-recurrent  if the (h)hv-torsion tensor
 $T$  satisfies the condition\\
  $(\nabla_{\gamma \overline{X}}T)(\overline{Y},\overline{Z})
 =\lambda_{o} (\overline{X}) T(\overline{Y},\overline{Z}).$

    \item[(c)]  $C^0$-recurrent  if the (h)hv-torsion tensor
 $T$  satisfies the condition\\
  $(D^{\circ}_{\gamma \overline{X}}T)(\overline{Y},\overline{Z})
 =\lambda_{o} (\overline{X}) T(\overline{Y},\overline{Z}).$

\end{description}
\end{defn}

\begin{thm}Let $(M,L)$ be a Finsler manifold  which admits a concurrent $\pi$-vector field $\overline{\zeta} $
such that $\lambda_{o} (\overline{\zeta})\neq0$. Then, the
following assertions are equivalent\,: \vspace{-0.2cm}
\begin{description}
\item[(a)] $(M,L)$ is a $C^h$-recurrent manifold,

\item[(b)]$(M,L)$ is a $C^v$-recurrent manifold,

\item[(c)]$(M,L)$ is a $C^0$-recurrent manifold,

\item[(d)]$(M,L)$ is a  Riemannian manifold.
\end{description}
\end{thm}

\begin{proof} It is to be noted that (b), (c) and (d) are equivalent
despite of the existence of a concurrent $\pi$-vector field
 \cite{r86}. The implication $\textbf{(d)} \Longrightarrow
\textbf{(a)}$ is trivial. It remains to prove that  $\textbf{(a)}
\Longrightarrow \textbf{(d)}$. Setting
$\overline{W}=\overline{\zeta}$ in (\ref{01}), making use of
$\widehat{P}(\overline{\zeta},\overline{X})=0=T(\overline{\zeta},\overline{X})$
(Corollary \ref{2.pp.2}), $P(\overline{X},\overline{Y},\overline{Z},
\overline{\zeta})=T(\overline{X},\overline{Y},\overline{Z})$
(Proposition \ref{2.pp.1}) and  $g(( \nabla_{\beta
\overline{Z}}T)(\overline{X}, \overline{Y}),\overline{W})=g((
\nabla_{\beta \overline{Z}}T)(\overline{X},
\overline{W}),\overline{Y})$ (Proposition 3.3 of \cite{r96}), we get
\begin{equation*}
   \nabla_{\beta \overline{\zeta}}T=0.
\end{equation*}
On the other hand,  Definition \ref{.def.1}(a) for
$\overline{X}=\overline{\zeta}$, yields
\begin{equation*}
   \nabla_{\beta \overline{\zeta}}T=\lambda_{o}(\overline{\zeta})T.
\end{equation*}
the  above two equations imply that $T=0$ and hence $(M,L)$ is
Riemannian.
\end{proof}

\begin{defn}\label{7.def.4}A Finsler manifold $(M,L)$ is said to be\,\emph{:}\vspace{-0.2cm}
\begin{description}
\item[(a)] quasi-$C$-reducible
if $dim(M)\geq 3$ and the Cartan tensor $T$ has the
from\vspace{-0.2cm}
  $$T(\overline{X},\overline{Y},\overline{Z})= A(\overline{X}
  ,\overline{Y})C(\overline{Z})+A(\overline{Y}
  ,\overline{Z})C(\overline{X})+A(\overline{Z}
  ,\overline{X})C(\overline{Y}),\vspace{-0.2cm}$$
where $A$ is a symmetric $\pi$-tensor field satisfying
$A(\overline{X},\overline{\eta})=0$.

 \item[(b)]  semi-$C$-reducible if $dim M \geq 3$ and the Cartan tensor $T$ has the form
  \begin{equation}\label{.7.eq.5}
  \begin{split}
    T(\overline{X},\overline{Y},\overline{Z})= & \frac{\mu}{n+1} \{\hbar(\overline{X}
  ,\overline{Y})C(\overline{Z})+\hbar(\overline{Y}
  ,\overline{Z})C(\overline{X})+\hbar(\overline{Z},\overline{X}) C(\overline{Y})\}\\
      &+ \frac{\tau}{C^{2}}C(\overline{X}) C(\overline{Y})
  C(\overline{Z}),\vspace{-0.2cm}
  \end{split}
  \end{equation}
 where \, $C^2:=C(\overline{C})\neq 0$,  $\mu$ and $\tau$ are scalar
functions  satisfying  $\mu +\tau=1$.
\item[(c)] $C$-reducible   if $dim M \geq 3$ and the
  Cartan tensor $T$ has the form
  \begin{equation}\label{.7.eq.6}
      T(\overline{X},\overline{Y},\overline{Z})=  \frac{1}{n+1} \{\hbar(\overline{X}
  ,\overline{Y})C(\overline{Z})+\hbar(\overline{Y}
  ,\overline{Z})C(\overline{X})+\hbar(\overline{Z},\overline{X})
  C(\overline{Y})\}.
   \end{equation}

\item[(d)] $C_{2}$-like   if $dim M \geq 2$ and  the Cartan tensor $T$ has the form
  \begin{equation*}
    T(\overline{X},\overline{Y},\overline{Z})= \frac{1}{C^{2}}C(\overline{X}) C(\overline{Y})
  C(\overline{Z}).
  \end{equation*}
  \end{description}
\end{defn}

\begin{prop}If a  quasi-$C$-reducible Finsler manifold $(M,L)$
 {\em($dim M \geq 3$)}  admits  a concurrent $\pi$-vector field, then
  $(M,L)$ is Riemannian, provided that $A(\overline{\zeta}, \overline{\zeta})\neq0$.
\end{prop}

\begin{proof} Follows from the defining property of quasi-$C$-reducibility
by setting $\overline{X}=\overline{Y}=\overline{\zeta}$ and using
the fact that  $C( \overline{\zeta})=0$ and
$A(\overline{\zeta},\overline{\zeta})\neq0$.
\end{proof}

\begin{thm}Let $(M,L)$ be a Finsler manifold {\em($dim M \geq 3$)}
which admits a concurrent $\pi$-vector field $\overline{\zeta} $,
then, we have\vspace{-0.2cm}
\begin{description}
   \item[(a)]  A $C$-reducible manifold $(M,L)$ is a Riemannian manifold.

  \item[(b)] A semi-$C$-reducible manifold $(M,L)$ is a $C_{2}$-like manifold.
\end{description}
\end{thm}

\begin{proof}~\par
\vspace{5pt} \noindent\textbf{(a)} Follows from the defining
property of $C$-reducibility by setting
$\overline{X}=\overline{Y}=\overline{\zeta}$, taking into account
Lemma \ref{.le.1}(e), Lemma \ref{.le.2} and $C(
\overline{\zeta})=0$.

\vspace{5pt}
 \noindent\textbf{(b)} Let $(M,L)$ be
semi-$C$-reducible. Setting $\overline{X}=
\overline{Y}=\overline{\zeta}$ and $\overline{Z}=\overline{C}$ in
(\ref{.7.eq.5}), taking into account Corollary \ref{2.pp.2}(a) and
$C(\overline{\zeta})=0$, we get
                  $$\mu\hbar(\overline{\zeta},\overline{\zeta})C(\overline{C})=0.$$
From which, since  $\hbar(\overline{\zeta},\overline{\zeta})\neq0$
and $C(\overline{C})\neq0$, it follows that $\mu=0$.  Consequently,
$(M,L)$ is $C_{2}$-like.
\end{proof}

\begin{defn} The condition\vspace{-0.2cm}
 \begin{equation}\label{.7.eq.3}
 \mathbb{T}(\overline{X},\overline{Y}, \overline{Z}, \overline{W} ):=
  L(\nabla_{\gamma \overline{X}}T)(\overline{Y}, \overline{Z}, \overline{W}) +
  \mathfrak{S}_{\overline{X}, \overline{Y}, \overline{Z}, \overline{W} }\,
   \ell(\overline{X}) T(\overline{Y}, \overline{Z}, \overline{W})=0
  \vspace{-0.2cm}
\end{equation}
will be  called the $\mathbb{T}$-condition.

The more relaxed condition\vspace{-0.2cm}
 \begin{equation}\label{.7.eq.4}
 \mathbb{T}_{o}(\overline{X},\overline{Y} ):=
  L(\nabla_{\gamma \overline{X}}C)(\overline{Y}) +
  \mathfrak{S}_{\overline{X}, \overline{Y}}\, \ell(\overline{X}) C(\overline{Y})=0\quad\qquad\,\,\quad\quad
  \end{equation}
will be  called the $\mathbb{T}_{o}$-condition.
\end{defn}

\begin{thm}Let $(M,L)$ be a Finsler manifold  which admits a concurrent $\pi$-vector field $\overline{\zeta}$. Then, the following assertions are
equivalent\,:\vspace{-0.2cm}
\begin{description}
   \item[(a)] $(M,L)$  satisfies the
  $\mathbb{T}$-condition,

  \item[(b)] $(M,L)$  satisfies  the $\mathbb{T}_{o}$-condition,

\item[(c)] $(M,L)$ is  Riemannian.
\end{description}
\end{thm}

\begin{proof}~\par
\vspace{5pt}
 \noindent\textbf{(a)\,$\Longrightarrow$(c):} Follows from  (\ref{.7.eq.3}) by setting
$\overline{W}=\overline{\zeta}$, taking into account that
$T(\overline{X},
\overline{\zeta})=T(\overline{\zeta},\overline{X})=0$ and
$\ell(\overline{\zeta})=\frac{B}{L}\neq0$.

\vspace{5pt}
 \noindent\textbf{(b)\,$\Longrightarrow$(c):} Follows from  (\ref{.7.eq.4}) by setting $\overline{X}=\overline{\zeta}$,
taking into account that $C( \overline{\zeta})=0$, $(\nabla_{\gamma
\overline{X}}C)(\overline{Y})=(\nabla_{\gamma
\overline{Y}}C)(\overline{X})$ and $\ell(\overline{\zeta})\neq0$.

\vspace{5pt}
 \noindent The other implications are trivial.
\end{proof}

\begin{defn}\label{3.def.3} A Finsler manifold $(M,L)$ is said to be
  $S_{3}$-like if $dim(M)\geq 4$ and the v-curvature tensor
$S$ has the form\,{\em:}\vspace{-0.2cm}
\begin{equation}\label{.eq.5}
    S(\overline{X},\overline{Y},\overline{Z},\overline{W})=
\frac{Sc^{v}}{(n-1)(n-2)} \{
\hbar(\overline{X},\overline{Z})\hbar(\overline{Y},\overline{W})-\hbar(\overline{X},\overline{W})
\hbar(\overline{Y},\overline{Z}) \}.
\end{equation}
\end{defn}

\begin{thm}If an  $S_{3}$-like manifold  $(M,L)$  {\em($dim M \geq 4$)}
 admits a concurrent $\pi$-vector field $\overline{\zeta}$,
then, the v-curvature tensor $S$ vanishes.
\end{thm}

\begin{proof}
  Setting $\overline{Z}=\overline{\zeta}$ in (\ref{.eq.5}), taking Proposition \ref{2.pp.1}
  into account, we immediately get
$$ \frac{Sc^{v}}{(n-1)(n-2)} \{
\hbar(\overline{X},\overline{\zeta})\hbar(\overline{Y},\overline{W})-\hbar(\overline{X},\overline{W})
\hbar(\overline{Y},\overline{\zeta}) \}=0$$ Taking the trace of the
above equation, we have
$$\frac{Sc^{v}}{(n-1)(n-2)} \{
(n-1)\hbar(\overline{X},\overline{\zeta})-\hbar(\overline{X},\overline{\zeta})
\}=\frac{Sc^{v}}{(n-1)} \hbar(\overline{X},\overline{\zeta})=0$$
From which, since $\hbar(\overline{X},\overline{\zeta})\neq0$ (Lemma
\ref{.le.1}),  the vertical scalar curvature  $Sc^{v}$ vanishes.
Now, again, from (\ref{.eq.5}), the result follows.
\end{proof}

\begin{defn}\label{.7.def.8} A Finsler manifold $(M,L)$, where  $dim M\geq 3$,
is said to be\,{\em :}\vspace{-0.1cm}
\begin{description}
  \item[(a)] $P_{2}$-like if the hv-curvature tensor $P$ has the
form\,{\em:}\vspace{-0.1cm}
\begin{equation}\label{02}
   P(\overline{X},\overline{Y},\overline{Z}, \overline{W})=
 \omega(\overline{Z})T (\overline{X},\overline{Y},\overline{W})
 -\omega(\overline{W})\, T(\overline{X},\overline{Y},\overline{Z}),\vspace{-0.1cm}
\end{equation}
 where $\omega$ is
  a $(1)\,\pi$-form {\em{(}positively homogeneous of degree $0$)}.

 \item[(b)]$P$-reducible
  if  the $\pi$-tensor field
 $\widehat{P}(\overline{X},\overline{Y},\overline{Z}):
  =g(\widehat{P}(\overline{X},\overline{Y}),\overline{Z})$ has the form\vspace{-0.1cm}
\begin{equation}\label{020}\widehat{P}(\overline{X},\overline{Y},\overline{Z})=\delta(\overline{X})\hbar (\overline{Y},\overline{Z})
  +\delta(\overline{Y})\hbar (\overline{X},\overline{Z})+ \delta(\overline{Z})\hbar
  (\overline{X},\overline{Y}),\vspace{-0.1cm}
  \end{equation}
  where $\delta$ is
  the $\pi$-form defined by  $\delta(\overline{X})=\frac{1}{n+1}(\nabla_{\beta \overline{\eta}}\,C)(\overline{X}).$
\end{description}
  \end{defn}

\begin{thm}
Let $(M,L)$ be a Finsler manifold {\em($dim M \geq 3$)} which admits
a concurrent $\pi$-vector field $\overline{\zeta} $, then, we
have\vspace{-0.2cm}
\begin{description}
   \item[(a)]  A $P_{2}$-like manifold $(M,L)$ is a Riemannian manifold,
   provided that $\omega(\overline{\zeta})\neq-1$.

  \item[(b)] A $P$-reducible manifold $(M,L)$  is a Landsberg manifold.
\end{description}

\end{thm}
\begin{proof}~\par
\vspace{5pt}
 \noindent\textbf{(a)} Setting $\overline{Z}=\overline{\zeta}$ in (\ref{02}), taking into account
Proposition \ref{2.pp.1} and Corollary \ref{2.pp.2}, we immediately
get $$\left(\omega(\overline{\zeta})+1\right)
T(\overline{X},\overline{Y})=0.$$
 Hence, the result follows.

\vspace{5pt}
 \noindent\textbf{(b)} Setting $\overline{X}= \overline{Y}=\overline{\zeta}$ in
(\ref{020}) and taking into account that $(\nabla_{\beta
\overline{\eta}}C)( \overline{\zeta})=0$, we get
$\hbar(\overline{\zeta},\overline{\zeta})(\nabla_{\beta
\overline{\eta}}C)( \overline{Z})=0$, with
$\hbar(\overline{\zeta},\overline{\zeta})\neq0$ (Lemma \ref{.le.1}).
Consequently, $\nabla_{\beta \overline{\eta}}C=0$. Hence, again,
from Definition \ref{.7.def.8}(b), the (v)hv-torsion tensor
$\widehat{P}=0$.
\end{proof}

\begin{defn}\label{.7.def.9} A Finsler manifold $(M,L)$, where  $dim M\geq 3$,
is said to be\,{\em :}\vspace{-0.1cm}
\begin{description}
  \item[(a)]$h$-isotropic  if there exists a scalar $k_{o}$ such that the
horizontal curvature tensor $R$ has the form\vspace{-0.2cm}
$$R(\overline{X},\overline{Y})\overline{Z}=k_{o} \{g(\overline{X},\overline{Z})
\overline{Y}-g(\overline{Y},\overline{Z})\overline{X} \}.$$

  \item[(b)]of scalar curvature if there exists a  scalar function $k: \T M \To
 \Real$ such that\vspace{-0.2cm} $$R(\overline{\eta},\overline{X},\overline{\eta},\overline{Y})=
  k L^{2} \hbar(\overline{X},\overline{Y}), \vspace{-0.2cm} $$ where
  $k$ called the scalar curvature.
\end{description}
\end{defn}

\begin{thm}Let $(M,L)$, $dim M\geq 3$, be an $h$-isotropic Finsler manifold
admitting  a concurrent $\pi$-vector field $\overline{\zeta}$, then
we have
 \begin{description}
    \item[(a)] The h-curvature tensor $R$ of the Cartan connection  vanishes.
   \item[(b)] If $(M,L)$ is a Berwald manifold, then  it is a flat Riemannian manifold.
\end{description}
\end{thm}

\begin{proof}~\par
\vspace{5pt}
 \noindent\textbf{(a)} From Definition \ref{.7.def.9}(a), we have
\begin{equation}\label{.eq.6}
    R(\overline{X},\overline{Y},\overline{Z},\overline{W})=k_{o} \{g(\overline{X},\overline{Z})
g(\overline{Y},\overline{W})-g(\overline{Y},\overline{Z})g(\overline{X},\overline{W})
\}.
\end{equation}
Setting $\overline{Z}=\overline{\zeta}$ and $\overline{X}=
\overline{m}$ and noting that
$R(\overline{X},\overline{Y})\overline{\zeta}=0$ (Proposition
\ref{2.pp.1}), we have
$$k_{o} \{g(\overline{m},\overline{\zeta})
g(\overline{Y},\overline{W})-g(\overline{Y},\overline{\zeta})g(\overline{m},\overline{W})
\}=0.$$ Taking the trace of this equation, we get
$$k_{o} (n-1)g(\overline{m},\overline{\zeta})=0.$$
From which, since
$g(\overline{m},\overline{\zeta})=g(\overline{m},\overline{m})\neq0$
(Lemma \ref{.le.1}) and $dim M\geq 3$,  the scalar $k_{o}$ vanishes.
Now, again, from (\ref{.eq.6}), the result follows.

\vspace{5pt}
 \noindent\textbf{(b)} Follows from (a), taking  Theorem
\ref{.thm.1} into account.
\end{proof}

\begin{thm}Let $(M,L)$, $dim M\geq 3$, be a Finsler manifold  of scalar curvature
  admitting  a concurrent $\pi$-vector field $\overline{\zeta}$. The
  following statements hold\,\emph{:}
 \begin{description}
   \item[(a)]The scaler curvature $k$ vanishes.
   \item[(b)] The  deviation tensor field $H$ \emph{(}$H(\overline{X}):=\widehat{R}(\overline{\eta},
    \overline{X})$\emph{)} vanishes.
   \item[(c)]The (v)h-torsion tensor $\widehat{R}$ of Cartan connection  vanishes.
   \item[(d)] The h-curvature tensor ${R}^{\circ}$ of Berwald connection vanishes.
\item[(e)] The horizontal distribution is completely integrable.
 \end{description}
\end{thm}

\begin{proof}~\par
\vspace{5pt}
 \noindent\textbf{(a)}  Follows from Definition \ref{.7.def.9}(b) by setting
$\overline{Y}=\overline{\zeta}$, taking into account  Proposition
\ref{2.pp.1}(g) and  Lemma \ref{.le.1}(e).

\vspace{5pt}
 \noindent\textbf{(b)} Follows from (a) and Definition \ref{.7.def.9}(b).

\vspace{5pt}
 \noindent\textbf{(c)}, \textbf{(d)} and \textbf{(e)} Follow from
Theorem 4.7 of \cite{r96}.
\end{proof}


\Section{Energy $\beta$-change and Cartan nonlinear connection
 (Barthel connection)}

In the present and the next sections  we consider a perturbation, by
a concurrent $\pi$-vector field $\overline{\zeta}$, of the energy
function $E=\frac{1}{2}L^{2}$ of a Finsler structure $L$.
\par
 Let $(M,L)$ be a Finsler manifold. Consider the change \vspace{-0.2cm}
 \begin{equation}\label{2.eq.2}
\widetilde{L}^{2}(x,y)=L^{2}(x,y)+ B^{2}(x,y)\, , \ \ \ with \ \
B:=g(\overline{\eta},\overline{\zeta})=\alpha(\overline{\eta}),\vspace{-0.3cm}
\end{equation}
  $\widetilde{L}$ defines a new Finsler structure on $M$. The Finsler structure
   $\widetilde{L}$ is said to be obtained  from the Finsler structure $L$ by  the $\beta$-change
    (\ref{2.eq.2}). The $\beta$-change
    (\ref{2.eq.2}) will be referred to as an \textbf{energy $\beta$-change} (as it can be written
    in the form $\widetilde{E}=E+\frac{1}{2}B^{2}$, where $E$ and $\widetilde{E}$ are the energy functions
    corresponding to the Lagrangians $L$ and $\widetilde{L}$ respectively).

\vspace{5pt}
 \par
 The following two lemmas are useful for subsequence use.\,\vspace{-0.2cm}
\begin{lem}\label{2.co.1}The function $B(x,y)$ given by {\em (\ref{2.eq.2})}
has the properties
\begin{description}
    \item[(a)]$ B=d_{J}E(\beta \overline{\zeta})$,\quad $d_{J}B=\alpha\circ\rho
    $,\quad  $d_{h}B=-(i_{\overline{\eta}}\,g)\circ\rho $.
       \item[(b)] $dd_{J}B^{2}(\gamma \overline{X}, \beta \overline{Y})=
       2 \alpha(\overline{X}) \alpha(\overline{Y})$, \, $dd_{J}B^{2}(\gamma \overline{X}, \gamma \overline{Y})=0$.
  \item[(c)] $dd_{J}B^{2}(\beta \overline{X}, \beta \overline{Y})=2(\alpha\wedge i_{\overline{\eta}}\,g)(\overline{X},\overline{Y})$.
   \end{description}
\end{lem}
\begin{proof}~\par
\vspace{5pt}
 \noindent \textbf{(a)} From  $g(\overline{\eta},\overline{\eta})=2
E^{2}$ and $\nabla g=0$, one can show that
 $$d_{J}E(X)=g(\rho X, \overline{\eta}), \  \forall  \ X \in \cpp.$$
 Setting $X =\beta \overline{\zeta}$, we obtain $ B=d_{J}E(\beta \overline{\zeta})  $.
\par
 On the other hand,\vspace{-0.2cm} $$d_{J}B(X)=JX \cdot B=J X\cdot g(\overline{\zeta},\overline{\eta})
 =g(\overline{\zeta}, \nabla_{JX}\overline{\eta})=g(\overline{\zeta},\rho X)=\alpha(\rho X).\vspace{-0.2cm}$$
 Similarly,\vspace{-0.2cm} $$d_{h}B(X)=hX \cdot B=hX\cdot g(\overline{\zeta},\overline{\eta})
 =g(\nabla_{hX}\,\overline{\zeta}, \overline{\eta})=-g(\rho X,\overline{\eta})=-i_{\overline{\eta}}g(\rho X).$$

\vspace{5pt}
 \noindent \textbf{(b)} Making use of (a), we  have
\begin{eqnarray*}
 dd_{J}B^{2}(\gamma \overline{X}, \beta \overline{Y}) &=& \gamma \overline{X} \cdot d_{J}B^{2}(\beta \overline{Y})
 -\beta \overline{Y} \cdot d_{J}B^{2}(\gamma \overline{X})-d_{J}B^{2}([\gamma \overline{X}, \beta \overline{Y}])\\
  &=&2\gamma \overline{X}\cdot (B \gamma \overline{Y} \cdot B)-2\beta \overline{Y}\cdot (B J\gamma \overline{X} \cdot B)
  -2B J[\gamma \overline{X}, \beta \overline{Y}]\cdot B \\
  &=&2\gamma \overline{X}\cdot (B g(\overline{Y},\overline{\zeta}))-2B g(\rho[\gamma \overline{X}, \beta \overline{Y}],\overline{\zeta}) \\
   &=& 2\set{g(\overline{X},\overline{\zeta}) g(\overline{Y},\overline{\zeta})+B\gamma \overline{X}\cdot g(\overline{Y},\overline{\zeta})}-
   2B g(\nabla_{\gamma \overline{X}} \overline{Y}-T(\overline{X},\overline{Y}),\overline{\zeta})  \\
 &=& 2\set{g(\overline{X},\overline{\zeta}) g(\overline{Y},\overline{\zeta})+B g(\nabla_{\gamma \overline{X}} \overline{Y},\overline{\zeta})}-
   2B g(\nabla_{\gamma \overline{X}} \overline{Y},\overline{\zeta})\\
   &=&2 \alpha(\overline{X}) \alpha(\overline{Y}).
\end{eqnarray*}
Similarly, $dd_{J}B^{2}(\gamma \overline{X}, \gamma
\overline{Y})=0$.

 \vspace{5pt}
\noindent\textbf{(c)} The proof is analogous to that of (b).
\end{proof}

\begin{lem}\label{2.pp.metric} Let $(M,L)$ be a Finsler manifold. Let $g$ be the Finsler
metric associated with $L$ and  $\nabla$  the Cartan connection on
$\pi^{-1}(TM)$. Then, the following relation holds \vspace{-0.1cm}
\begin{equation}\label{2.eq.3}
   g(\overline{X},\overline{Y})=\Omega(\gamma \overline{X},\beta \overline{Y}) \quad \text{for
all}\quad \overline{X}, \overline{Y} \in \cp, \vspace{-0.1cm}
\end{equation}
where $\Omega:= dd_{J}E$ and  $\beta$ is the horizontal map
associated with $\nabla$.
\end{lem}

\begin{proof}
 Using the relations $d_{J}E(X)=g(\rho X, \overline{\eta})$, $\nabla g=0$, $T(\rho
X,\overline{\eta})=0$ and
$g({T}(\overline{X},\overline{Y}),\overline{Z})=
g({T}(\overline{X},\overline{Z}),\overline{Y})$, we get
\begin{eqnarray*}
  \Omega(\gamma \overline{X},\beta \overline{Y})&=&
  \gamma \overline{X}\cdot d_{J}E(\beta \overline{Y})-\beta \overline{Y}\cdot d_{J}E(\gamma \overline{X})-d_{J}E([\gamma \overline{X},\beta \overline{Y}])\\
&=&\gamma \overline{X}\cdot
g(\overline{Y},\overline{\eta})-g(\rho[\gamma \overline{X},\beta \overline{Y}],\overline{\eta})\\
&=&g(\nabla_{\gamma
\overline{X}}\overline{Y},\overline{\eta})+g(\overline{Y},\nabla_{\gamma
\overline{X}}\overline{\eta})-g(\rho[\gamma \overline{X},\beta \overline{Y}],\overline{\eta}) \\
&=& g(\overline{X},\overline{Y})+g(
T(\overline{X},\overline{Y}),\overline{\eta})=g(\overline{X},\overline{Y}),
\end{eqnarray*}
which proves the required relation.
\end{proof}
\par
The following result gives the relationship between $g$ and
$\widetilde{g}$.\vspace{-0.2cm}
\begin{prop}\label{2.pp.3}Under the energy $\beta$-change {\em (\ref{2.eq.2})}, the
Finsler metrics $g$ and $\widetilde{g}$ are related by
\vspace{-0.2cm}
\begin{equation}\label{.2.eq.3}
\widetilde{g}(\overline{X},\overline{Y})=g(\overline{X},\overline{Y})+\alpha
(\overline{X}) \alpha(\overline{Y}), \ \ for \ all \ \ \overline{X},
\overline{Y} \in \cp\, ,
    \end{equation}
    $\alpha$ being  the $\pi$-form  associated with $\overline{\zeta}$ under the
    duality defined by the metric $g$.
    \end{prop}

   \begin{proof}
   The proof follows  by applying the operator
    $\frac{1}{2}dd_{J}$ on both sides of (\ref{2.eq.2}), taking into
    account (\ref{2.eq.3}), Proposition  \ref{pp.1} and Lemma
    \ref{2.co.1}(b).\\
    In more detail,
    \begin{eqnarray*}
    \widetilde{g}(\overline{X},\overline{Y}) &=&dd_{J}\widetilde{E}(\gamma \overline{X},\widetilde{\beta} \overline{Y}) , \\
    &=&dd_{J}\widetilde{E}(\gamma \overline{X},{\beta} \overline{Y})-dd_{J}\widetilde{E}(\gamma \overline{X},\L {\beta} \overline{Y}) \\
    &=&dd_{J}E(\gamma \overline{X},\beta \overline{Y})+\frac{1}{2}dd_{J}B^{2}(\gamma \overline{X},{\beta} \overline{Y}),
    \quad as\ dd_{J}\widetilde{E}(\gamma \overline{X},\gamma \overline{Y})=0\\
    &=&g(\overline{X},\overline{Y})+\alpha(\overline{X}) \alpha(\overline{Y}).
    \end{eqnarray*}
   \end{proof}

    \begin{cor}Under the energy $\beta$-change {\em (\ref{2.eq.2})},
    the exterior 2-forms  $\Omega$ and $\widetilde{\Omega}$ are
    related by\vspace{-0.2cm}
    $$\widetilde{\Omega}= \Omega+\frac{1}{2}\,dd_{J}B^{2}.$$
    \end{cor}

\begin{thm}\label{..2.th.2} Let $(M,L)$ and $(M,\widetilde {L})$ be
  two Finsler manifolds related by the energy $\beta$-change {\em (\ref{2.eq.2})}.  The associated
  canonical sprays $G$ and $\widetilde{G}$ are related by
\vspace{-0.2cm}
$$\widetilde{G}= G+\frac{L^{2}}{1+p^{2}}\,\gamma \overline{\zeta} \,; \ \ p^{2}:=g(\overline{\zeta},\overline{\zeta}).$$
\end{thm}
\begin{proof}
 From the above Corollary, we have
$$\widetilde{\Omega}= \Omega+\frac{1}{2}\,dd_{J}B^{2}. \vspace{-0.2cm}$$
As the difference between two sprays is a vertical vector field,
assume that $\widetilde{G}=G+\gamma \overline{\mu}$, for some
$\overline{\mu}\in\cp$, then we have
\begin{equation}\label{eq}
     \left.
    \begin{array}{rcl}
-d\widetilde{E}(X)&=&i_{\widetilde{G}}\,\widetilde{\Omega}(X)=i_{G+\gamma
\overline{\mu}}\,\widetilde{\Omega}(X)\\
&=&i_{G}\,\Omega(X)+\frac{1}{2}i_{G}\,dd_{J}B^{2}(X)+i_{\gamma
\overline{\mu}}\,\Omega(X)+\frac{1}{2}i_{\gamma
\overline{\mu}}\,dd_{J}B^{2}(X).{\ \ \ \ }
 \end{array}
  \right.
 \end{equation}
Now, we compute the terms on the right hand side (using Lemma
\ref{2.co.1} and Lemma \ref{2.pp.metric}):
\begin{eqnarray*}
  i_{G}\,\Omega(X)&=&-dE(X). \\
  i_{\gamma \overline{\mu}}\,\Omega(X) &=& \Omega(\gamma \overline{\mu}, X)=
  \Omega(\gamma \overline{\mu},\gamma K X+\beta \rho X)=
  \Omega(\gamma \overline{\mu},\beta \rho X)=
  g(\overline{\mu}, \rho X).\\
          \frac{1}{2}\,i_{G}\,dd_{J}B^{2}(X)
&=&\frac{1}{2}\{dd_{J}B^{2}(\beta \overline{\eta} ,\beta \rho
X)-dd_{J}B^{2}(\gamma KX, \beta \overline{\eta})\}\\
 &=&\{g(\overline{\eta}, \overline{\zeta})g(\rho X, \overline{\eta})-g(\overline{\eta}, \overline{\eta})g(\rho X,\overline{\zeta})\}-g(K X,\overline{\zeta})g(\overline{\eta}, \overline{\zeta})\\
&=&\{g(\overline{\eta}, \overline{\zeta})g(\rho X,
\overline{\eta})-g(\overline{\eta}, \overline{\eta})g(\rho X,\overline{\zeta})\}-\{X\cdot g(\overline{\eta},\overline{\zeta})+g(\overline{\eta}, \rho X)\}g(\overline{\eta}, \overline{\zeta})\\
&=&-L^{2}g(\overline{\zeta}, \rho X)-B dB(X). \\
\frac{1}{2}\,i_{\gamma
\overline{\mu}}\,dd_{J}B^{2}(X)&=&\frac{1}{2}\,dd_{J}B^{2}(\gamma
\overline{\mu}, \beta \rho X+ \gamma K X )=
\frac{1}{2}\,dd_{J}B^{2}(\gamma \overline{\mu}, \beta \rho X)=
g(\overline{\zeta},\overline{\mu})g(\overline{\zeta},\rho X).
\end{eqnarray*}
From these,  together with $ d\widetilde{E}(X)=dE(X)+BdB(X),$
Equation (\ref{eq}) reduces to \vspace{-0.1cm}
\begin{equation}\label{2.eq.5}
   g(\overline{\mu}, \rho X)+g(\overline{\zeta},\overline{\mu})g(\overline{\zeta},\rho X)-L^{2}g(\overline{\zeta}, \rho X)=0.\vspace{-0.2cm}
\end{equation}
Consequently,
\begin{equation*}
    \overline{\mu}=(L^{2}-g(\overline{\zeta},\overline{\mu}))\,\overline{\zeta}.
\end{equation*}
Again, from (\ref{2.eq.5}), by setting  $X=\beta \overline{\zeta}$,
we obtain\vspace{-0.2cm}
\begin{equation*}\label{2.eq.6}
   g(\overline{\zeta},\overline{\mu})=\frac{L^2 g(\overline{\zeta},\overline{\zeta})}{1+g(\overline{\zeta},\overline{\zeta})}
   .\vspace{-0.3cm}
\end{equation*}
 Therefore,
$$\overline{\mu} =\frac {L^{2}}{1+g(\overline{\zeta},\overline{\zeta})}\,\overline{\zeta}.$$
Hence  the result.
\end{proof}

\begin{thm}\label{2.th.2} Let $(M,L)$ and $(M,\widetilde {L})$ be
  two Finsler manifolds related by the energy $\beta$-change {\em (\ref{2.eq.2})}.  The associated
  Barthel connections $\Gamma$ and $\widetilde{\Gamma}$  are related by
\vspace{-0.2cm}
\begin{equation*}\label{2.eq.4}
         \widetilde{\Gamma} =\Gamma+ 2\,\frac{d_{J}E}{1+p^{2}}\otimes \gamma
\overline{\zeta} \,.
\end{equation*}
 \end{thm}

\begin{proof}
From Theorem \ref{..2.th.2} and the formula \cite{r20}
 $$[f X, J]=f[X,J]+df\wedge i_{X}J-d_{J}f\otimes X,$$
 we obtain \vspace{-0.2cm}
\begin{eqnarray*}
  \widetilde{\Gamma} &=& [J,\widetilde{G}] = \left[J, G+\frac{L^{2}}{1+p^{2}}\gamma \overline{\zeta}\right] =
  [J,G]-\left[\frac{L^{2}}{1+p^{2}}\gamma \overline{\zeta},J\right]\\
&=& [J,G]-\frac{L^{2}}{1+p^{2}}[\gamma
\overline{\zeta},J]-d(\frac{L^{2}}{1+p^{2}})\wedge i_{\gamma
\overline{\zeta}}J+d_{J}(\frac{L^{2}}{1+p^{2}})\otimes\gamma
\overline{\zeta}.\vspace{-0.2cm}
\end{eqnarray*}
Now, \vspace{-0.2cm}
\begin{eqnarray*}
   d_{J}(\frac{L^{2}}{1+p^{2}})&=&\frac{2d_{J}E}{(1+p^{2})}, \ \ as \
   d_{J}p^{2}=0, \text{by Theorem \ref{th.ind}},
   \vspace{-0.2cm}
\end{eqnarray*}
whereas,\vspace{-0.2cm}
\begin{eqnarray*}
   [\gamma \overline{\zeta}, J]X &=&[\gamma \overline{\zeta}, JX]-J[\gamma \overline{\zeta}, X] \\
  &=&\gamma \{\nabla_{\gamma \overline{\zeta}}\,\rho X-\nabla_{J X}\,\overline{\zeta}\}
  -\gamma\{\nabla_{\gamma \overline{\zeta}}\,\rho X-T(\overline{\zeta},\rho X)\}=0,
\end{eqnarray*}
 by Lemma  \ref{bracket} and Corollary \ref{2.pp.2}.
 Consequently,\\
$$\widetilde{\Gamma}= \Gamma +\frac{2\,d_{J}E}{1+g(\overline{\zeta}, \overline{\zeta})}\otimes \gamma \overline{\zeta},$$
which proves the result.
\end{proof}

\begin{rem}\label{.rem} Comparing Theorem \ref{2.th.2} and Proposition
\ref{pp.1}, we find that\vspace{-0.2cm}
\begin{equation*}
         \widetilde{\Gamma} =\Gamma - 2{\L} \ ,\vspace{-0.2cm}
\end{equation*}
where\vspace{-0.2cm}
\begin{equation}\label{2.eq.4}
     {\L} =-\displaystyle{\frac{d_{J}E}{1+p^{2}}}\otimes \gamma
\overline{\zeta} .
\end{equation}
\end{rem}

\begin{lem}\label{lem} Let $ {\L} $ be the 1-form  defined by (\ref{2.eq.4}).
 The following formulae hold\,\emph{:}
\begin{description}
  \item[(a)]$[{\L} X,{\L} Y]=\mathfrak{U}_{(X,Y)}\set{\displaystyle{\frac{g(\rho X,\overline{\eta})}{(1+p^{2})^{2}}}\,
  \{g(\nabla_{\gamma \overline{\zeta}}\rho Y,\overline{\eta})+g(\rho Y,\overline{\zeta})\}\,\gamma
  \overline{\zeta}}$.
  \item[(b)]$[h X,{\L} Y]=\displaystyle{-\frac{g(\rho Y,\overline{\eta})}{1+p^{2}}[hX,\gamma \overline{\zeta}]-
  \frac{g(\nabla_{h X}\rho Y,\overline{\eta})}{1+p^{2}}\,\gamma \overline{\zeta}-
  \frac{2g(\rho Y,\overline{\eta})g(\rho X,\overline{\zeta})}{(1+p^{2})^{2}}\,\gamma
  \overline{\zeta}}$,
\end{description}
where $\mathfrak{U}_{(X,Y)}\{A(X,Y)\}=A(X,Y)-A(Y,X)$.
\end{lem}
\begin{proof} We prove (a) only; the proof of (b) is similar.
 Since  ${\L}(X)=-\frac{g(\rho X,\overline{\eta})}{1+p^{2}}\,\gamma \overline{\zeta}$, we have
\begin{eqnarray*}
  [{\L} X,{\L} Y] &=& \left[\frac{g(\rho X,\overline{\eta})}{1+p^{2}}\gamma \overline{\zeta},
  \frac{g(\rho Y,\overline{\eta})}{1+p^{2}}\gamma \overline{\zeta}\right]=\frac{g(\rho X,\overline{\eta})g(\rho Y,\overline{\eta})}{(1+p^{2})^{2}}[\gamma \overline{\zeta},\gamma \overline{\zeta}]
   \\
   &&+\frac{g(\rho X,\overline{\eta})}{1+p^{2}}\set{\gamma \overline{\zeta}\cdot\frac{g(\rho Y,\overline{\eta})}{1+p^{2}}}\gamma \overline{\zeta}
   -\frac{g(\rho Y,\overline{\eta})}{1+p^{2}}\,
   \set{\gamma \overline{\zeta}\cdot\frac{g(\rho X,\overline{\eta})}{1+p^{2}}}\,\gamma \overline{\zeta}.
\end{eqnarray*}
The result follows from this relation together with the following
expressions
\begin{eqnarray*}
  [\gamma \overline{\zeta},\gamma \overline{\zeta}] &=& \gamma\{ \nabla_{\gamma \overline{\zeta}}\,\overline{\zeta}-\nabla_{\gamma \overline{\zeta}}\,\overline{\zeta}\}=0, \\
  \gamma \overline{\zeta}\cdot\set{\frac{g(\rho X,\overline{\eta})}{1+p^{2}}} &=&
   \frac{\gamma \overline{\zeta}\cdot g(\rho X,\overline{\eta})}{1+p^{2}} =
   \frac{g(\nabla_{\gamma \overline{\zeta}}\rho X,\overline{\eta})+g(\rho X,\overline{\zeta})}{1+p^{2}},
\end{eqnarray*}
by Theorem \ref{th.ind}.
\end{proof}

\begin{prop}\label{2.pp.4} Under the energy $\beta$-change {\em
(\ref{2.eq.2})}, the curvature tensors $\widetilde{\Re}$ and $\Re$
of the associated  Barthel connections $\widetilde{\Gamma}$ and
$\Gamma$ are related by\vspace{-0.2cm}
\begin{equation}\label{2.eq.8}
\widetilde{\Re}(X,Y)=\Re(X,Y)+ \mathfrak{U}_{X,Y}\{\mathbb{{H}}(X,
Y)\}\,,
\end{equation}
where $\mathbb{{H}}(X, Y):=\displaystyle{\frac{g(\rho Y,
\overline{\eta})}{1+p^2}}J X+\displaystyle{\frac{g(\rho X,
\overline{\eta})g(\rho Y, \overline{\zeta})}{(1+p^2)^2}}\,\gamma
\overline{\zeta}$.
\end{prop}
\begin{proof}
 By the identity  $\Re(X,Y)=-v[h X,h Y]$
  \,\cite{r97}, we have\vspace{-0.1cm}
  $$\widetilde{\Re}(X,Y)=-\widetilde{v}\left[\widetilde{h} X,\widetilde{h}
  Y\right].\vspace{-0.1cm}$$
  Using Proposition \ref{pp.1} and the fact that ${\L}[{\L} X,{\L}
  Y]=0$, we get\vspace{-0.2cm}
\begin{equation}\label{2.eq.9}
\widetilde{\Re}(X,Y)=\Re(X,Y)-[{\L} X,{\L}
Y]-{\L}[hX,hY]+\mathfrak{U}_{X,Y} \{
    v[hX,{\L} Y]+{\L}[hX,{\L} Y]\}.\vspace{-0.2cm}
\end{equation}
The result follows from (\ref{2.eq.9}) and the following
formulae which can be easily computed using Lemma \ref{lem} and Theorem \ref{2.th.2}\\
 $[{\L} X,{\L} Y]=\displaystyle{\mathfrak{U}_{(X,Y)}\set{\frac{g(\rho X,\overline{\eta})}{(1+p^{2})^{2}}
  \{g(\nabla_{\gamma \overline{\zeta}}\rho Y,\overline{\eta})+g(\rho Y,\overline{\zeta})\}\gamma
  \overline{\zeta}}}$,\\
\noindent
 $v[h X,{\L} Y]=-\displaystyle{\frac{g(\rho Y, \overline{\eta})}{1+p^2}}\,v[h
X,\gamma \overline{\zeta}]-\displaystyle{\frac{2g(\rho X,
\overline{\zeta})g(\rho Y,
\overline{\eta})}{(1+p^2)^2}}\,\gamma \overline{\zeta}-\displaystyle{\frac{g(\nabla_{h X}\rho Y,\overline{\eta})}{1+p^2}}\,\gamma \overline{\zeta}$,\\

\noindent
 ${\L} [h X,{\L} Y]=\displaystyle{\frac{g(\rho Y, \overline{\eta})}{(1+p^2)^{2}}}\,g(\rho [h
X,\gamma \overline{\zeta}],\overline{\eta})\gamma
  \overline{\zeta}$,\\

\noindent
 ${\L} [h X,h Y]=-\displaystyle{\frac{g(\rho [h
X,hY],\overline{\eta})}{1+p^2}}\,\gamma \overline{\zeta}$.
\end{proof}

\begin{cor}under the energy $\beta$-change {\em (\ref{2.eq.2})}, the curvature tensors $\Re$
is invariant if and only if, the vector 2-form $\mathbb{H}$ is
symmetric.
\end{cor}


\Section{Fundamental connections under an energy $\beta$-change}

In this section, we investigate the transformation of the
fundamental linear connections of Finsler geometry, as well as their
curvature tensors, under the energy $\beta$-change {(\ref{2.eq.2})}.
We start our investigation with the Cartan connection.

\vspace{5pt}
 \par
The following lemmas are useful for subsequent use.\vspace{-0.2cm}
\begin{lem}{\em \cite{r92}}\label{3.le.1} Let $(M,L)$ be a Finsler manifold and  $g$ the Finsler
metric associated with $L$. The Cartan connection $\nabla$
 is completely determined by  the  relations\,\emph{:}
    \begin{description}
     \item[(a)]  $ 2g(\nabla _{vX}\rho Y,\rho Z) =vX\cdot g(\rho Y,\rho Z)+
     g(\rho Y,\rho [hZ,vX])+g(\rho Z,\rho [vX,hY])$.

     \item[(b)]  $ 2g(\nabla _{hX}\rho Y,\rho Z)  = hX\cdot g(\rho Y,\rho
                Z)+ hY\cdot g(\rho Z,\rho X)-hZ\cdot g(\rho X,\rho
                Y)$\\
  $  -g(\rho X,\rho [hY,hZ])+g(\rho Y,\rho
[hZ,hX])+g(\rho Z,\rho
    [hX,hY])$.
 \end{description}
\end{lem}

Since the Cartan connection $\nabla$ has the same horizontal and
vertical projectors  as  the Barthel connection $\Gamma=[J,G]$
\cite{r92}, we get\vspace{-0.2cm}
\begin{lem}\label{3.le.2}Under the energy $\beta$-change {\em
(\ref{2.eq.2})}, we have \vspace{-0.3cm}
$$ \widetilde{h}=h-{\L}, \ \   \widetilde{v}=v+{\L} \ \ \ with \
 {\L} :=-{\frac{d_{J}E}{1+g(\overline{\zeta},\overline{\zeta})}}\otimes \gamma \overline{\zeta}. $$
\end{lem}

\begin{thm}\label{3.th.1} Let  $(M,L)$ and $(M,\widetilde {L})$ be two Finsler
manifolds  related by the energy $\beta$-change {\em
(\ref{2.eq.2})}. Then the associated Cartan  connections $\nabla$
and
  $\widetilde{\nabla}$  are related by:\vspace{-0.2cm}
\begin{equation}\label{13.eq.r1}
\widetilde {\nabla} _{X}\rho {Y} =
    \nabla _{X}\rho {Y}-\displaystyle{\frac{g(\rho X,\rho Y)}{1+p^{2}}\,  \overline{\zeta}}\, ,
    \ \ with\ \ p^2:=g(\overline{\zeta},\overline{\zeta}) . \vspace{-0.2cm}
 \end{equation}
\noindent In particular,
\begin{description}
  \item[(a)]  $ \widetilde {\nabla}_{\gamma \overline{X}}\overline{Y}=
    {\nabla}_{ \gamma \overline{X}}\overline{Y}$,

  \item[(b)] $\widetilde {\nabla} _{\widetilde{\beta}\overline{X}}\overline{Y} =
    \nabla _{\beta \overline{X}}\overline{Y}-U( \overline{X},
    \overline{Y}),$\ \ with \
    $U( \overline{ X}, \overline{Y})=
 \displaystyle{\frac{g(\overline{X},\overline{Y})}{1+p^{2}}\,\overline{\zeta}}-\frac{g(\overline{X}, \overline{\eta})}{1+p^{2}}\nabla_{\gamma \overline{\zeta}}\overline{Y}$.
\end{description}
  \end{thm}

\begin{proof}
Using  Lemma  \ref{3.le.1}(a),  Proposition
 \ref{2.pp.3} and  Lemma \ref{3.le.2}, noting  the fact that
$\rho[\widetilde{h}Z,\widetilde{v}X]=\rho[{h}Z,\widetilde{v}X]$ (as
$\rho[\L Z,\widetilde{v}X]=0$), we get\vspace{-0.35cm}

\begin{eqnarray*}
  2\widetilde{g}(\widetilde {\nabla} _{\widetilde {v}X}\rho Y,\rho Z)
   &=&\{vX\cdot g(\rho Y,\rho Z) + g(\rho Y,\rho [hZ,vX]) +g(\rho Z,\rho
 [vX,hY])\}\\
 &&+\{{\L} X\cdot g(\rho Y,\rho Z)+ g(\rho Y,\rho [hZ,{\L} X])+g(\rho Z,\rho [{\L}
 X,hY])\}\\
 &&+\{g(\nabla _{vX}\rho Y,\overline{\zeta})g(\rho Z,\overline{\zeta})+g(\rho Y,\overline{\zeta})g\left(\nabla _{vX}\rho Z,\overline{\zeta}\right)\\
 &&+g(\rho Y,\overline{\zeta}) g(\textbf{T}(vX,hZ)-\nabla_{vX}\rho Z,\overline{\zeta})\\
&&+g(\rho Z,\overline{\zeta}) g\left(\nabla_{vX}\rho Y-\textbf{T}(vX,hY),\overline{\zeta}\right)\}\\
&&+\{g\left(\nabla _{{\L} X}\rho Y,\overline{\zeta}\right)g(\rho
Z,\overline{\zeta})+g(\rho Y,\overline{\zeta})
g\left(\nabla _{{\L} X}\rho Z,\overline{\zeta}\right)\\
 &&+g(\rho Y,\overline{\zeta}) g(\textbf{T}({\L} X,hZ)-\nabla_{{\L} X}\rho Z,\overline{\zeta})\\
&&+g(\rho Z,\overline{\zeta}) g\left(\nabla_{{\L} X}\rho
Y-\textbf{\textbf{T}}({\L} X,hY),\overline{\zeta}\right)\}.
\end{eqnarray*}
Since $\textbf{T}({\L} X,hY)=-\frac{g(\overline{X},
\overline{\eta})}{1+p^{2}}T(\overline{\zeta}, \rho Y)=0$ and
$g(\textbf{T}(vX,hY),\overline{\zeta})=-g(T(K
X,\overline{\zeta}),\rho Y)=0$, the above relation takes the form
\begin{eqnarray*}
  2\widetilde{g}(\widetilde {\nabla} _{\widetilde {v}X}\rho Y,\rho Z)
  &=&2g(\nabla _{vX}\rho Y,\rho Z) +2g(\nabla _{{\L} X}\rho Y,\rho Z)
\\
&&+2g(\rho Z,\overline{\zeta})g(\nabla _{vX}\rho Y,\overline{\zeta})+2g(\rho Z,\overline{\zeta})g(\nabla_{{\L} X}\rho Y,\overline{\zeta})\\
&=& 2g(\nabla_{\widetilde{v} X}\rho Y,\rho
Z)+2g(g(\nabla_{\widetilde{v} X}\rho
Y,\overline{\zeta})\,\overline{\zeta},\rho Z).
\end{eqnarray*}
Consequently,
 \begin{equation}\label{3.eq.1}
    {g}(\widetilde {\nabla} _{\widetilde {v}X}\rho Y,\rho Z)
      +g(g(\widetilde{\nabla}_{\widetilde{v} X}\rho Y,\overline{\zeta})\,\overline{\zeta},\rho Z) =
       g(\nabla_{\widetilde{v} X}\rho Y,\rho Z)+g(g(\nabla_{\widetilde{v} X}\rho Y,\overline{\zeta})\,\overline{\zeta},\rho Z)\vspace{-0.1cm}.
 \end{equation}
From which, by setting $Z=\beta \overline{\zeta}$, we
get\vspace{-0.1cm}
 \begin{equation}\label{3.eq.2}
   {g}(\widetilde {\nabla} _{\widetilde {v}X}\rho Y,\overline{\zeta})
       -g(\nabla_{\widetilde{v} X}\rho Y,\overline{\zeta})=0\vspace{-0.1cm}
     .
 \end{equation}
Then, Equations (\ref{3.eq.1}) and (\ref{3.eq.2}) imply
that\vspace{-0.1cm}
 \begin{equation}\label{3.eq.3}
   \widetilde {\nabla} _{\widetilde {v}X}\rho Y
       =\nabla_{\widetilde{v} X}\rho Y     .\vspace{-0.1cm}
 \end{equation}
\par
Similarly, by  Lemma \ref{3.le.1}(b),  Proposition
  \ref{2.pp.3} and   Lemma \ref{3.le.2}, noting
   that $T(\overline{\zeta}, \overline{X})=T(\overline{X}, \overline{\zeta})=0$, we get after long but easy calculations \vspace{-0.1cm}
\begin{eqnarray*}
   2\widetilde{g}(\widetilde {\nabla} _{\widetilde {h}X}\rho Y,\rho Z)
  &=&2g(\nabla_{h X}\rho Y,\rho Z)-2g(\nabla_{{\L} X}\rho Y,\rho Z)
+2\{g(\rho X,\rho[{\L} Y,hZ])\\
&&-g(\rho X,\nabla_{{\L} Y}\rho Z)\}-2\{g(\rho Y,\rho[{\L} Z,hX])-g(\rho Y,\nabla_{{\L} Z}\rho X)\}\\
&&+2g(g(\nabla_{h X}\rho Y,\overline{\zeta})\,\overline{\zeta},\rho
Z)-2g(g(\nabla_{{\L} X}\rho
Y,\overline{\zeta})\,\overline{\zeta},\rho Z)-2g(\rho Z,\overline{\zeta})g(\rho X, \rho Y)\\
&=&2g(\nabla_{\widetilde{h} X}\rho Y,\rho
Z)+2g(g(\nabla_{\widetilde{h} X}\rho
Y,\overline{\zeta})\,\overline{\zeta},\rho Z)-2g(\rho
Z,\overline{\zeta})g(\rho X, \rho Y)\\
&&+2g(\rho Y,\textbf{T}({\L} Z,h X))-2g(\rho X,\textbf{T}({\L} Y, hZ))\\
&=&2g(\nabla_{\widetilde{h} X}\rho Y,\rho
Z)+2g(g(\nabla_{\widetilde{h} X}\rho
Y,\overline{\zeta})\,\overline{\zeta},\rho Z)-2g(\rho
Z,\overline{\zeta})g(\rho X, \rho Y).
\end{eqnarray*}
Consequently,
\begin{equation}\label{3.eq.4}
\begin{array}{rcl}
     &&g(\widetilde {\nabla} _{\widetilde {h}X}\rho Y,\rho Z)
     +g(g(\widetilde{\nabla}_{\widetilde {h}X}\rho Y,\overline{\zeta})\,\overline{\zeta},\rho Z)
       = \\
     &=&g(\nabla_{\widetilde{h} X}\rho Y,\rho Z)
     +g(g(\nabla_{\widetilde{h} X}\rho Y,\overline{\zeta})\,\overline{\zeta},\rho Z)-g(\rho Z,\overline{\zeta})g(\rho X, \rho Y).\vspace{-0.1cm}
     \end{array}
\end{equation}
From which,  setting $Z=\beta \overline{\zeta}$, we
get\vspace{-0.2cm}
 \begin{equation}\label{3.eq.5}
   \widetilde{g}(\widetilde {\nabla} _{\widetilde {h}X}\rho Y,\overline{\zeta})
       -g(\nabla_{\widetilde{h} X}\rho Y,\overline{\zeta})
       =-\frac{g(\overline{\zeta},\overline{\zeta})g(\rho X,\rho Y)}{1+g(\overline{\zeta},\overline{\zeta})}\ .\vspace{-0.2cm}
    \end{equation}
Then, Equations (\ref{3.eq.4}) and (\ref{3.eq.5}) imply that
\vspace{-0.1cm}
 \begin{equation}\label{3.eq.6}
   \widetilde {\nabla} _{\widetilde {h}X}\rho Y
       =\nabla_{\widetilde{h} X}\rho Y-
       \frac{g(\rho X,\rho Y)}{1+g(\overline{\zeta},\overline{\zeta})}\,\overline{\zeta}     .\vspace{-0.1cm}
 \end{equation}
\par
Now, (\ref{13.eq.r1}) follows from   (\ref{3.eq.3}) and
(\ref{3.eq.6}).
\end{proof}

\vspace{7pt}
\par
As a consequence of Theorem \ref{3.th.1} and Definition
\ref{2.def.1}, we have \vspace{-0.2cm}
\begin{cor}If a Finsler manifold admits a concurrent $\pi$-vector field $\overline{\zeta}$, then the vector
field $\overline{\zeta}$ is no more concurrent  with respect to the
transformed  metric {\em(\ref{2.eq.2})}.
\end{cor}

\begin{cor} \label{3.pp.1} Under the energy $\beta$-change  {\em
(\ref{2.eq.2})}, we have\,:
\begin{description}

    \item[(a)] The  maps  $
  \overline{Y} \longmapsto \nabla_{\gamma \overline{X}}\overline{Y}$   and
   $ \overline{Y} \longmapsto \nabla_{\gamma \overline{Y}}\overline{X}$
  are  invariant.

    \item[(b)] The (h)hv-torsion $T$ and  the Cartan tensor are  invariant.

\end{description}
\end{cor}

\begin{thm}\label{3.th.2} Let $(M,L)$ and
$(M,\widetilde{L})$ be two Finsler manifolds related by the energy
$\beta$-change {\em (\ref{2.eq.2})}. The curvature tensors of the
associated Cartan connections $\nabla$ and $\widetilde{\nabla}$ are
related by\,{\em:}
\begin{equation}\label{3.eq.7}
   \widetilde{\textbf{K}}(X,Y)\overline{Z}  = \textbf{K}(X,Y)\overline{Z}
  + \frac{g( \textbf{T}(X,Y),\overline{Z})}{1+p^2}\,\overline{\zeta}+H(\rho{X},
  \rho{Y})\overline{Z},
\end{equation}
\noindent where $H$ is the $\pi$-tensor field defined by
  \begin{equation}
        H(\rho{X}, \rho{Y})\overline{Z} =
     \mathfrak{U}_{X,Y} \set{\frac{g(\rho X,\overline{Z})}{1+p^2}\,\rho
     Y +\frac{g(\rho Y,\overline{Z})g(\rho X,\overline{\zeta})}{(1+p^2)^2}\,\overline{\zeta}}.
       \end{equation}
 \noindent  In particular,
 \begin{description}
  \item[(a)] $\widetilde{S}(\overline{X},\overline{Y})\overline{Z}=
S(\overline{X},\overline{Y})\overline{Z}$,

  \item[(b)] $ \widetilde{P}(\overline{X},\overline{Y})\overline{Z}=
  P(\overline{X},\overline{Y})\overline{Z}-
  \displaystyle{\frac{g( T(\overline{Y},\overline{X}),\overline{Z})}{1+p^2}}\,\overline{\zeta},$

  \item [(c)]
 $\widetilde{R}(\overline{X}, \overline{Y})\overline{Z} = R(\overline{X},\overline{Y})\overline{Z}
 +H(\overline{X},\overline{Y})\overline{Z}$.
\end{description}
\end{thm}

\begin{proof}
 By Theorem \ref{3.th.1}, we have: \vspace{-0.2cm}
\begin{eqnarray*}
  \widetilde{\nabla}_{X}\widetilde{\nabla}_{Y}\overline{Z}
       &=& \nabla_{X}\nabla_{Y}\overline{Z}-\frac{g(\rho X,\nabla_{Y}\overline{Z})}{1+p^2}\,\overline{\zeta}+\frac{g(\rho Y,\overline{Z})}{1+p^2}\rho X \\
   & & - \frac{g(\rho Y,\overline{Z})g(\rho X,\overline{\zeta})}{(1+p^2)^2}\,\overline{\zeta}
   -\frac{g(\nabla_{X}\rho {Y},\overline{Z})}{1+p^2}\,\overline{\zeta}-\frac{g(\rho Y,\nabla_{X}\overline{Z})}{1+p^2}\,\overline{\zeta},\\
   \widetilde{\nabla}_{[X,Y]}\overline{Z}&=&\displaystyle{\nabla_{[X,Y]}\overline{Z}
-\frac{g(\nabla_{X}\rho {Y},\overline{Z})}{1+p^2}\,\overline{\zeta}
+\frac{g(\nabla_{Y}\rho {X},\overline{Z})}{1+p^2}\,\overline{\zeta}
+\frac{g(\textbf{T}(X,Y),\overline{Z})}{1+p^2}\,\overline{\zeta}}.\vspace{-0,2cm}
\end{eqnarray*}
The result follows from the above identities and the definition of
$\widetilde{\textbf{K}}$.
\end{proof}

\par
In view of the above theorem, we have \vspace{-0.2cm}
\begin{cor}\label{3.pp.2}
 The v-curvature tensor $S$ and
 the (v)hv-torsion tensor  $\widehat{P}$ are
invariant under  the energy $\beta$-change  {\em (\ref{2.eq.2})}.
\end{cor}

\vspace{0.2cm} Now, we turn our attention to  the Chern (Rund)
connection $ D^{\diamond}. $\vspace{-0.2cm}
\begin{thm}{\em\cite{r94}}\label{4.th.1} Let $(M,L)$ be a Finsler manifold.
The Chern connection $ D^{\diamond} $ is expressed in terms of the
Cartan  connection as\vspace{-0.2cm}
  $$ D^{\diamond}_{X}\overline{Y} = \nabla _{X}\overline{Y}
- T(KX,\overline{Y}), \quad \forall\: X \in \mathfrak{X} (T M), \,
\overline{Y} \in \mathfrak{X}(\pi(M)). \vspace{-0.2cm}$$ In
particular, we have:
\begin{description}
  \item[(a)] $ D^{\diamond}_{\gamma \overline{X}}\overline{Y}=\nabla _{\gamma
  \overline{X}}\overline{Y}-T(\overline{X},\overline{Y})$.

 \item[(b)] $ D^{\diamond}_{\beta \overline{X}}\overline{Y}=\nabla _{\beta
  \overline{X}}\overline{Y}$.
\end{description}
\end{thm}

\begin{lem} If  $\overline{\zeta} \in \cp$ is
a concurrent $\pi$-vector field, then we have  \vspace{-0.2cm}
\begin{equation}\label{2.eq.1}
     D^{\diamond}_{\beta \overline{X}}\,\overline{\zeta}=- \overline{X} ,\quad D^{\diamond}_{\gamma \overline{X}}\,\overline{\zeta}=0 .\vspace{-0.2cm}
   \end{equation}
\end{lem}

\begin{proof}
 The proof follows from Definition \ref{2.def.1} and the above
theorem, taking into account the fact that
$T(\overline{\zeta},\overline{X})=T(\overline{X},\overline{\zeta})=0$,
for all $\overline{X}\in \cp$.
\end{proof}

\vspace{5pt}
\par
Using the above Lemma, we get\vspace{-0.2cm}
\begin{prop}\label{4.pp.1}Let $\overline{\zeta} \in \cp$ be
a concurrent $\pi$-vector field on $(M,L)$. \\
The hv-curvature tensor $P^{\diamond}$ has the
properties:\vspace{-0.2cm}
\begin{description}
      \item[(a)]
      $P^{\diamond}(\overline{X},\overline{Y})\,\overline{\zeta}=
      P^{\diamond}(\overline{X},\overline{\zeta})\overline{Y}=0$.

   \item[(b)] $(D^{\diamond}_{\gamma \overline{Z}}P^{\diamond})(\overline{X},\overline{Y}, \overline{\zeta})=0$.

  \item[(c)] $(D^{\diamond}_{\beta \overline{Z}}P^{\diamond})(\overline{X},\overline{Y},
  \overline{\zeta})= P(\overline{X},\overline{Y}) \overline{Z}$.

  \item[(d)]$(D^{\diamond}_{\beta \overline{\zeta}}P^{\diamond})(\overline{X},\overline{Y},
    \overline{\zeta})=0$.
\end{description}
The h-curvature tensor $R^{\diamond}$ has the
properties:\vspace{-0.2cm}
\begin{description}
 \item[(e)]$R^{\diamond}(\overline{X},\overline{Y})\,\overline{\zeta}=0$.

    \item[(f)]$(D^{\diamond}_{\gamma \overline{Z}}R^{\diamond})(\overline{X},\overline{Y},
      \overline{\zeta})=0$.

 \item[(g)]$(D^{\diamond}_{\beta \overline{Z}}R^{\diamond})(\overline{X},\overline{Y}, \overline{\zeta})=R^{\diamond}(\overline{X},\overline{Y})\overline{Z}$.

   \item[(h)] $(D^{\diamond}_{\beta\overline{\zeta}}R^{\diamond})(\overline{X},\overline{Y}, \overline{\zeta})=0$.
\end{description}
\end{prop}

From Theorems  \ref{4.th.1} and \ref{3.th.1}, tacking into account
the fact that the (h)hv-torsion $T$ is invariant under
(\ref{2.eq.2}) and
$T(\overline{\zeta},\overline{X})=T(\overline{X},\overline{\zeta})=0$,
we have\vspace{-0.2cm}

\begin{thm}\label{4.th.2} Let  $(M,L)$ and $(M,\widetilde {L})$ be two Finsler
manifolds  related by the energy $\beta$-change {\em
(\ref{2.eq.2})}. Then the associated Chern connections
$D^{\diamond}$ and
  $\widetilde{D^{\diamond}}$  are related by:\vspace{-0.2cm}
\begin{equation}\label{3.eq.r1}
\widetilde{D^{\diamond}} _{X}\rho {Y} =
    D^{\diamond} _{X}\rho {Y}-\displaystyle{\frac{g(\rho X,\rho Y)}{1+p^{2}}\,\overline{\zeta}} . \vspace{-0.2cm}
 \end{equation}
\noindent In particular,
\begin{description}

  \item[(a)]  $ \widetilde {D^{\diamond}}_{\gamma \overline{X}}\overline{Y}=
    {D^{\diamond}}_{ \gamma \overline{X}}\overline{Y}$,

  \item[(b)] $\widetilde {D^{\diamond}} _{\widetilde{\beta}\overline{X}}\overline{Y} =
    D^{\diamond} _{\beta \overline{X}}\overline{Y}-U^{\diamond}( \overline{X},
    \overline{Y}),$\ \ with \
    $U^{\diamond}( \overline{ X}, \overline{Y})=
 \displaystyle{\frac{g(\overline{X},\overline{Y})}{1+p^{2}}\,\overline{\zeta}}-\frac{g(\overline{X}, \overline{\eta})}{1+p^{2}}D^{\diamond}_{\gamma \overline{\zeta}}\overline{Y}$.
\end{description}
  \end{thm}

Concerning the curvature tensors of the Chern connection, we
have\vspace{-0.2cm}
\begin{thm}\label{4.th.3} Let $(M,L)$ and
$(M,\widetilde{L})$ be two Finsler manifolds  related by the energy
$\beta$-change {\em (\ref{2.eq.2})}. The curvature tensors of the
associated Chern connections $D^{\diamond}$ and
$\widetilde{D^{\diamond}}$ are related by\,{\em:}\vspace{-0.2cm}
\begin{equation*}\label{4.eq.7}
  \widetilde{{\textbf{K}^{\diamond}}}(X,Y)\overline{Z}={\textbf{K}^{\diamond}}(X,Y)\overline{Z}
  + \frac{g( \textbf{T}(X,Y),\overline{Z})}{1+p^2}\,\overline{\zeta}+ {H}^{\diamond}(\rho{X},
  \rho{Y})\overline{Z},
\end{equation*}
\noindent where $H^{\diamond}$ is the $\pi$-tensor field defined by
  \begin{equation*}
        {H}^{\diamond}(\rho{X}, \rho{Y})\overline{Z} =
   \mathfrak{U}_{X,Y} \set{\frac{g(\rho X,\overline{Z})}{1+p^2}\rho
   Y -\frac{g(\rho X,T(K Y,\overline{Z}))}{1+p^2}\,\overline{\zeta}+
  \frac{g(\rho Y,\overline{Z})g(\rho X, \overline{\zeta})}{(1+p^2)^2}\,\overline{\zeta}}.
      \end{equation*}

 \noindent  In particular,
 \begin{description}
  \item[(a)] $\widetilde{S^{\diamond}}(\overline{X},\overline{Y})\overline{Z}=
S^{\diamond}(\overline{X},\overline{Y})\overline{Z}=0$,

  \item[(b)] $ \widetilde{P^{\diamond}}(\overline{X},\overline{Y})\overline{Z}=
  P^{\diamond}(\overline{X},\overline{Y})\overline{Z}-2\,
  \displaystyle{\frac{g( T(\overline{Y},\overline{X}),\overline{Z})}{1+p^2}}\,\overline{\zeta}
 ,$

  \item [(c)]
 $\displaystyle{\widetilde{R^{\diamond}}(\overline{X}, \overline{Y})\overline{Z} = R^{\diamond}(\overline{X},\overline{Y})\overline{Z}
 +\mathfrak{U}_{\overline{X},\overline{Y}} \set{
    \frac{g(\overline{X},\overline{Z})}{1+p^2}\,\overline{Y}
          +\frac{g(\overline{Y},\overline{Z})g(\overline{X},\overline{\zeta})}
          {(1+p^2)^2}\,\overline{\zeta}}}$.
\end{description}
\end{thm}

\vspace{0.2cm} Now, we investigate the effect of the energy
$\beta$-change (\ref{2.eq.2}) on the Hashiguchi connection $
D^{*}$.\vspace{-0.2cm}
\begin{thm}{\em\cite{r94}}\label{5.th.1} Let $(M,L)$
be a Finsler manifold. The Hashiguchi connection $ D^{*} $ is
expressed in terms of the Cartan  connection $\nabla $
as\vspace{-0.2cm}
  $$ D^{*}_{X}\overline{Y} = \nabla _{X}\overline{Y}
+\widehat{P}(\rho X,\overline{Y}),  \quad \forall\: X \in
\mathfrak{X} (T M), \, \overline{Y} \in \mathfrak{X}(\pi(M)) .
\vspace{-0.2cm}$$ In particular, we have:
\begin{description}
  \item[(a)] $ D^{*}_{\gamma \overline{X}}\overline{Y}=\nabla _{\gamma
  \overline{X}}\overline{Y}$.

 \item[(b)] $ D^{*}_{\beta \overline{X}}\overline{Y}=\nabla _{\beta
  \overline{X}}\overline{Y}+ \widehat{P}(\overline{X},\overline{Y}).$
\end{description}
\end{thm}
 As $\widehat{P}(\overline{X},\overline{\zeta})=0$, we get
\begin{lem}\label{5.le.1} If $\overline{\zeta} \in \cp$ is
a concurrent $\pi$-vector field, then  it satisfies the following
properties\vspace{-0.2cm}
\begin{equation}\label{2.eq.1}
     D^{*}_{\beta \overline{X}}\,\overline{\zeta}=- \overline{X} , \quad
    D^{*}_{\gamma \overline{X}}\,\overline{\zeta}=0 .
    \end{equation}
\end{lem}

\begin{prop}\label{4.pp.1}Let $\overline{\zeta} \in \cp$ be
a concurrent $\pi$-vector field on $(M,L)$. \\
The v-curvature tensor $S^{*}$ satisfies the properties:
\begin{description}
 \item[(a)]$S^{*}(\overline{X},\overline{Y})\,\overline{\zeta}=0$.
 \item[(b)]$(D^{*}_{\gamma \overline{Z}}S^{*})(\overline{X},\overline{Y}, \overline{\zeta})=0$.
 \item[(c)]$(D^{*}_{\beta \overline{Z}}S^*)(\overline{X},\overline{Y}, \overline{\zeta})=S^{*}(\overline{X},\overline{Y})\overline{Z}$.
  \end{description}
The hv-curvature tensor $P^{*}$ satisfies the  properties\,:
\begin{description}
    \item[(d)] $P^{*}(\overline{X},\overline{Y})\,\overline{\zeta}=-T(\overline{Y},\overline{X})$,\quad $P^{*}(\overline{X},\,\overline{\zeta})\overline{Y}=0$ .
       \item[(e)] $(D^{*}_{\gamma \overline{Z}}P^{*})(\overline{X},\overline{Y},\overline{\zeta})=-(D^{*}_{\gamma \overline{Z}}T)(\overline{Y},\overline{X})$.
  \item[(f)] $(D^{*}_{\beta \overline{Z}}P^{*})(\overline{X},\overline{Y},\overline{\zeta})=-(D^{*}_{\beta \overline{Z}}T)(\overline{Y},\overline{X})+
  P^{*}(\overline{X},\overline{Y})\overline{Z}$.
 \end{description}
The h-curvature tensor $R^{*}$ satisfies the  properties\,:
\begin{description}
 \item[(g)]$R^{*}(\overline{X},\overline{Y})\,\overline{\zeta}=0$.
    \item[(h)]$(D^{*}_{\gamma \overline{Z}}R^{*})(\overline{X},\overline{Y}, \overline{\zeta})=0$.
 \item[(l)]$(D^{*}_{\beta \overline{Z}}R^{*})(\overline{X},\overline{Y}, \overline{\zeta})=R^{*}(\overline{X},\overline{Y})\overline{Z}$.
\end{description}
\end{prop}

From Theorem \ref{5.th.1} and Theorem \ref{3.th.1}, tacking into
account the fact that the (v)hv-torsion tensor $\widehat{P}$ is
invariant and
${\widehat{P}}(\overline{\zeta},\overline{X})=\widehat{P}(\overline{X},\overline{\zeta})=0$,
we have\vspace{-0.2cm}

\begin{thm}\label{5.th.2} Let  $(M,L)$ and $(M,\widetilde {L})$ be two Finsler
manifolds related by the energy $\beta$-change {\em (\ref{2.eq.2})}.
Then the associated Hashiguchi connections $D^*$ and
  $\widetilde{D^{*}}$  are related by:\vspace{-0.2cm}
\begin{equation}\label{3.eq.r1}
\widetilde{D^{*}} _{X}\rho {Y} =
    D^{*} _{X}\rho {Y}-\displaystyle{\frac{g(\rho X,\rho Y)}{1+p^{2}}\,\overline{\zeta}}. \vspace{-0.2cm}
 \end{equation}
\noindent In particular,
\begin{description}
  \item[(a)]  $ \widetilde {D^{*}}_{\gamma \overline{X}}\overline{Y}=
    {D^{\diamond}}_{ \gamma \overline{X}}\overline{Y}$,

  \item[(b)] $\widetilde {D^{*}} _{\widetilde{\beta}\overline{X}}\overline{Y} =
    D^{*} _{\beta \overline{X}}\overline{Y}-U^{*}( \overline{X},
    \overline{Y}),$\ \ with \
    $U^{*}( \overline{ X}, \overline{Y})=
 \displaystyle{\frac{g(\overline{X},\overline{Y})}{1+p^{2}}\,\overline{\zeta}}-\frac{g(\overline{X}, \overline{\eta})}{1+p^{2}}D^{*}_{\gamma \overline{\zeta}}\overline{Y}.$
\end{description}
  \end{thm}


\begin{thm}\label{5.th.3} Let $(M,L)$ and
$(M,\widetilde{L})$ be two Finsler manifolds related by the energy
$\beta$-change {\em (\ref{2.eq.2})}. The curvature tensors of the
associated Hashiguchi connections ${D}^{*}$ and
$\widetilde{{D}^{*}}$ are related by\,{\em:}\vspace{-0.2cm}
\begin{equation*}\label{5.eq.7}
   \widetilde{{\textbf{K}}^{*}}(X,Y)\overline{Z}  ={ \textbf{K}}^{*}(X,Y)\overline{Z}
  + \frac{g( \textbf{T}(X,Y),\overline{Z})}{1+p^2}\,\overline{\zeta}
  +{H}^{*}(\rho X,\rho Y)\overline{Z},
\end{equation*}
\noindent where ${H}^{*}$ is the $\pi$-tensor field defined
by\vspace{-0.2cm}
\begin{equation*}
{H}^{*}(\rho X,\rho Y)\overline{Z}= \mathfrak{U}_{X,Y}
\set{\frac{g(\rho X,\overline{Z})}{1+p^2}\rho
   Y+\frac{g(\rho Y,\overline{Z})g(\rho X, \overline{\zeta})}{(1+p^2)^2}\,\overline{\zeta}}.
\end{equation*}
 \noindent  In particular,
 \begin{description}
  \item[(a)] $\widetilde{{S}^{*}}(\overline{X},\overline{Y})\overline{Z}=
S^{*}(\overline{X},\overline{Y})\overline{Z}$,

  \item[(b)] $ \widetilde{{P}^{*}}(\overline{X},\overline{Y})\overline{Z}=
  {P}^{*}(\overline{X},\overline{Y})\overline{Z}-
  \displaystyle{\frac{g( T(\overline{Y},\overline{X}),\overline{Z})}{1+p^2}}\,\overline{\zeta}
 ,$

  \item [(c)]
 $\widetilde{{R}}^{*}(\overline{X}, \overline{Y})\overline{Z} = {R}^{*}(\overline{X},\overline{Y})\overline{Z}
 +{H}^{*}(\overline{X},\overline{Y})\overline{Z}.$
\end{description}
\end{thm}


\vspace{0.2cm}
 We terminate our study by the fourth fundamental connection in Finsler geometry, namely the Berwald connection
 $D^{\circ}$.

By Theorem \ref{2.th.1}, as
${P}(\overline{X},\overline{\zeta})=T(\overline{X},\overline{\zeta})=0$,
we have\vspace{-0.2cm}
\begin{lem} If $\overline{\zeta} \in \cp$ is
a concurrent $\pi$-vector field, then \vspace{-0.2cm}
\begin{equation}\label{2.eq.1}
     D^{\circ}_{\beta \overline{X}}\,\overline{\zeta}=- \overline{X} , \quad D^{\circ}_{\gamma \overline{X}}\,\overline{\zeta}=0 .
    \end{equation}
\end{lem}

\begin{prop}\label{6.pp.1}Let $\overline{\zeta} \in \cp$ be
a concurrent $\pi$-vector field on $(M,L)$.\\
The hv-curvature tensor $P^{\circ}$ has the
properties\emph{\,:}\vspace{-0.2cm}
\begin{description}
      \item[(a)] $P^{\circ}(\overline{X},\overline{Y})\,\overline{\zeta}=P^{\circ}(\overline{X},\overline{\zeta})\overline{Y}=0$.

   \item[(b)] $(D^{\circ}_{\gamma \overline{Z}}P^{\circ})(\overline{X},\overline{Y}, \overline{\zeta})=0$.
  \item[(c)] $(D^{\circ}_{\beta \overline{Z}}P^{\circ})(\overline{X},\overline{Y},
   \overline{\zeta})= P^{\circ}(\overline{X},\overline{Y})\overline{Z}$.
 \end{description}
The h-curvature tensor $R^{\circ}$ has the
properties\emph{\,:}\vspace{-0.2cm}
\begin{description}
 \item[(d)]$R^{\circ}(\overline{X},\overline{Y})\,\overline{\zeta}=0$.
    \item[(e)]$(D^{\circ}_{\gamma \overline{Z}}R^{\circ})(\overline{X},\overline{Y}, \overline{\zeta})=0$.
 \item[(f)]$(D^{\circ}_{\beta \overline{Z}}R^{\circ})(\overline{X},\overline{Y}, \overline{\zeta})=R^{\circ}(\overline{X},\overline{Y})\overline{Z}$.
\end{description}
\end{prop}

From Theorem \ref{2.th.1} and Theorem \ref{3.th.1},  taking into
account that the torsion tensors $T$ and $\widehat{P}$ are invariant
under (\ref{2.eq.2}),
$\widehat{P}(\overline{\zeta},\overline{X})=\widehat{P}(\overline{X},\overline{\zeta})=0$
and
$T(\overline{\zeta},\overline{X})=T(\overline{X},\overline{\zeta})=0$,
we obtain\vspace{-0.2cm}

\begin{thm}\label{6.th.1} Let  $(M,L)$ and $(M,\widetilde {L})$ be two Finsler
manifolds related by the energy $\beta$-change {\em (\ref{2.eq.2})}.
Then the associated Berwald connections $D^{\circ}$ and
$\widetilde{D^{\circ}}$  are related by\,\emph{:}\vspace{-0.2cm}
\begin{equation}\label{6.eq.r1}
\widetilde{D^{\circ}} _{X}\rho {Y} =
    D^{\circ} _{X}\rho {Y}-\displaystyle{\frac{g(\rho X,\rho Y)}{1+p^{2}}\,\overline{\zeta}}. \vspace{-0.2cm}
 \end{equation}
\noindent In particular,
\begin{description}

  \item[(a)]  $ \widetilde {D^{\circ}}_{\gamma \overline{X}}\overline{Y}=
    {D^{\circ}}_{ \gamma \overline{X}}\overline{Y}$,

  \item[(b)] $\widetilde {D^{\circ}} _{\widetilde{\beta}\overline{X}}\overline{Y} =
    D^{\circ} _{\beta \overline{X}}\overline{Y}-{U}^{\circ}( \overline{X},
    \overline{Y}),$\ \ with \
    ${U}^{\circ}( \overline{ X}, \overline{Y})=
 \displaystyle{\frac{g(\overline{X},\overline{Y})}{1+p^{2}}\,\overline{\zeta}}-\frac{g(\overline{X}, \overline{\eta})}{1+p^{2}}D^{\circ}_{\gamma \overline{\zeta}}\overline{Y}$.
\end{description}
  \end{thm}


\begin{thm}\label{6.th.3} Let $(M,L)$ and
$(M,\widetilde{L})$ be two Finsler manifolds related by the energy
$\beta$-change {\em (\ref{2.eq.2})}. The curvature tensors of the
associated Berwald connections $D^{\circ}$ and
$\widetilde{D^{\circ}}$ are related by\,{\em:}\vspace{-0.2cm}
\begin{equation*}\label{6.eq.7}
  \widetilde{{\textbf{K}}^{\circ}}(X,Y)\overline{Z}  = { \textbf{K}}^{\circ}(X,Y)\overline{Z}
  + \frac{g( \textbf{T}(X,Y),\overline{Z})}{1+p^2}\,\overline{\zeta}+{H}^{\circ}(\rho {X}, \rho
  {Y})\overline{Z},
\end{equation*}
\noindent where ${H}^{\circ}$ is the $\pi$-tensor field defined
by\vspace{-0.2cm}

$${H}^{\circ}(\rho {X}, \rho {Y})\overline{Z} = \mathfrak{U}_{X,Y} \{\frac{g(\rho X,\overline{Z})}{1+p^2}\rho
   Y-\frac{g(\rho X,T(K Y,\overline{Z}))}{1+p^2}\,\overline{\zeta}+
   \frac{g(\rho Y,\overline{Z})g(\rho X, \overline{\zeta})}{(1+p^2)^2}\,\overline{\zeta} \}.$$

 \noindent  In particular,
 \begin{description}
  \item[(a)] $\widetilde{{S}^{\circ}}(\overline{X},\overline{Y})\overline{Z}=
S^{\circ}(\overline{X},\overline{Y})\overline{Z}=0$,

  \item[(b)] $ \widetilde{{P}^{\circ}}(\overline{X},\overline{Y})\overline{Z}=
  {P}^{\circ}(\overline{X},\overline{Y})\overline{Z}-2\,
  \displaystyle{\frac{g( T(\overline{Y},\overline{X}),\overline{Z})}{1+p^2}}\,\overline{\zeta}
 ,$

  \item [(c)]
 $\displaystyle{\widetilde{{R}^{\circ}}(\overline{X}, \overline{Y})\overline{Z} = {R}^{\circ}(\overline{X},\overline{Y})\overline{Z}
 + \mathfrak{U}_{\overline{X},\overline{Y}} \set{\frac{g(\overline{X},\overline{Z})}{1+p^2}\,\overline{Y}
              +\frac{g(\overline{Y},\overline{Z})g(\overline{X},\overline{\zeta})}{(1+p^2)^2}\,\overline{\zeta}}}$.
\end{description}
\end{thm}


\vspace{0.8cm} \noindent{\bf \Large {Concluding remarks}}

\vspace{9pt}

\par
On the present work, we have the following comments and remarks:

\begin{itemize}
  \item The fundamental $\pi$-vector field $\overline{\eta}$ and
a concurrent $\pi$-vector field $\overline{\zeta}$ have different
properties, namely
\begin{eqnarray*}
   \nabla_{\gamma \overline{X}}\,\overline{\eta} = \overline{X}\, , \,&
   &\, \nabla_{\gamma \overline{X}}\,\overline{\zeta} = 0,\\
   \nabla_{\beta \overline{X}}\,\overline{\eta} = 0\, , \,&
    & \,\nabla_{\beta \overline{X}}\,\overline{\zeta} = -\overline{X}.
\end{eqnarray*}
 Moreover, an important difference is that  $\overline{\eta}$ is dependent on the directional
 argument $y$, whereas
  $\overline{\zeta}$  is independent of the directional argument $y$ as has been proved.

\item Although  there are  differences between the fundamental $\pi$-vector field
$\overline{\eta}$ and a concurrent $\pi$-vector field
$\overline{\zeta}$, the $(h)hv$-torsion tensor $T$ of $\nabla$ has
the common properties:

\begin{equation*}
  T(\overline{X},\overline{\eta})=T(\overline{\eta},
  \overline{X})=0;\,\,\,\,
  T(\overline{X},\overline{\zeta})=T(\overline{\zeta},\overline{X})=0.
\end{equation*}

Moreover, the $(v)hv$-torsion tensor $\widehat{P}$ of $\nabla$ has
the properties:
\begin{equation*}
  \widehat{P}(\overline{X},\overline{\eta})=\widehat{P}(\overline{\eta}, \overline{X})=0;\,\,\,\,
  \widehat{P}(\overline{X},\overline{\zeta})=\widehat{P}(\overline{\zeta},\overline{X})=0.
\end{equation*}


\item Special Finsler spaces play an important role in Finsler geometry. For this
reason, we studied the effect of the existence of a concurrent
$\pi$-vector field on  some important special Finsler spaces. An
interesting  result obtained is that  many of these special Finsler
spaces admitting a concurrent
 $\pi$-vector field are equivalent to a Riemannian space.

\item Randers spaces, which are obtained by the $\beta$-change
 $\widetilde{L}
 (x,y)=L
 (x,y)+ B
 (x,y),$ where $B:=b_{i}(x)y^{i}$,  are very important in physical applications \cite{ny3}. In this paper, we
investigate  intrinsically an energy $\beta$-change (similar in form
to a Randers change),
$$\widetilde{L}^{2}(x,y)=L^{2}(x,y)+ B^{2}(x,y),$$
 where $B:=g(\overline{\zeta},\overline{\eta})$, $\overline{\zeta}$ being a concurrent
 $\pi$-vector field.
 Under this change, the
torsion tensors $T$ and $\widehat{P}$ are invariant,
$\widehat{P}(\overline{\zeta},\overline{X})=\widehat{P}(\overline{X},\overline{\zeta})=0$
and
$T(\overline{\zeta},\overline{X})=T(\overline{X},\overline{\zeta})=0$.
Consequently,  the difference tensors of the fundamental Finsler
 connections, namely,   the Cartan connection, the Berwald connection,
the Chern connection and the  Hashiguchi connection have not only
simple but also similar forms.
\end{itemize}

\providecommand{\bysame}{\leavevmode\hbox
to3em{\hrulefill}\thinspace}
\providecommand{\MR}{\relax\ifhmode\unskip\space\fi MR }
\providecommand{\MRhref}[2]{%
  \href{http://www.ams.org/mathscinet-getitem?mr=#1}{#2}
} \providecommand{\href}[2]{#2}

\end{document}